\documentclass[12pt,epsfig,amsfonts]{amsart} 
\usepackage{amsmath,amsthm,amssymb,amscd,epsfig}
\usepackage{graphicx}

\newtheorem*{theorema}{Main Theorem}

\newtheorem{lemma}{Lemma}[section]

\newtheorem{remark}[lemma]{Remark}
\newtheorem{prop}[lemma]{Proposition}
\newtheorem{cor}[lemma]{Corollary}

\begin{document}
\author{Hiroki Takahasi}

\address{Department of Mathematics,
Keio University, Yokohama,
223-8522, JAPAN} 
\email{hiroki@math.keio.ac.jp}


\title[Removal of phase transition of the Chebyshev quadratic] 
{Removal of phase transition of the Chebyshev quadratic
and thermodynamics of H\'enon-like maps\\ near the first bifurcation}
\date{\today}

\begin{abstract}
We treat a problem 
at the interface of dynamical systems and equilibrium statistical physics.
It is well-known that the geometric pressure function $$t\in\mathbb R\mapsto 
\sup_{\mu}\left\{h_\mu(T_2)-t\int\log |dT_2(x)|d\mu(x)\right\}$$ of the
Chebyshev quadratic map $T_2(x)=1-2x^2$ $(x\in\mathbb R)$ is not differentiable at $t=-1$.
We show that this phase transition can be ``removed", by an arbitrarily small singular perturbation of the map $T_2$ into H\'enon-like diffeomorphisms.
A proof of this result relies on an elaboration of the well-known inducing techniques adapted to H\'enon-like dynamics near the first bifurcation.
\end{abstract}

\maketitle


\section{Introduction}
The {\it thermodynamic formalism}, i.e., 
the formalism of equilibrium statistical physics developed by G. W. Gibbs and others,
 has been successfully brought into the ergodic theory of chaotic dynamical systems (see e.g., \cite{Bow75, Rue78} and the references therein).
In the classical setting, it deals with a continuous map $f$ of a
compact metric space $X$ and a continuous function $\varphi$ on $X$,
and looks for {\it equilibrium measures}
which maximize (the minus of) the free energy $F_\varphi(\mu)=h_\mu(f)+\int \varphi d\mu$
among all $f$-invariant Borel probability measures on $X$. 
A relevant problem is to study the regularity of {\it the pressure function}
$t\in\mathbb R\mapsto P(t\varphi)$, where $P(t\varphi)=\sup_{\mu}F_{t\varphi}(\mu)$.

The existence and uniqueness of equilibrium measures depends upon details of the system and the potential.
For transitive uniformly hyperbolic systems and H\"older continuous potentials,
 the existence and uniqueness of equilibrium measures as well as
the analyticity of the pressure function
has been established in the pioneering works of Bowen, Ruelle and Sinai
\cite{Bow75,Rue78,Sin72}.
The latter property is interpreted as {\it the lack of phase transition}.

One important problem in dynamics is to understand  
structurally unstable, or nonhyperbolic systems \cite{PalTak93}.
The main problem which equilibrium statistical physics
tries to clarify is that of phase transitions \cite{Rue78}. 
Hence, it is natural to
study how phase transitions are affected by small perturbations of dynamics.

A natural candidate for a potential is the so-called {\it geometric potential} $\varphi(x)=-\log \|Df|E^u_x\|$,
where $E^u_x$ denotes the unstable direction at $x$ which reflect the chaotic behavior of $f$.
For nonhyperbolic systems, 
$x\mapsto E^u_x$ is often merely measurable, 
and may even be unbounded as in the case of one-dimensional maps with critical points.
These defects sometimes lead to the occurrence of phase transitions, e.g., 
the loss of analyticity or differentiability of the pressure function.
Typically, at phase transitions, there exist multiple equilibrium measures.


As an emblematic example, 
consider the family of quadratic maps $T_a\colon x\in\mathbb R\mapsto 
1-ax^2$ $(a>-1/4)$
and the associated family of geometric pressure functions
$ t\in\mathbb R\mapsto P(-t\log |dT_a|)$ given by
 \begin{equation}\label{pressure1}P(-t\log |dT_a|)=\sup_{\mu}\left\{h_\mu(T_a)-t\int\log |dT_a|d\mu\right\}.\end{equation}
Here, $h_\mu(T_a)$ denotes the Kolmogorov-Sinai 
entropy of $(T_a,\mu)$ and the supremum is taken over all
$T_a$-invariant Borel probability measures.

For $a>2$, the Julia set does not contain the critical point $x=0$, and so 
the dynamics is uniformly hyperbolic and structurally stable. According to the classical theory,
for any $t\in\mathbb R$ there exists a unique equilibrium measure for the potential $-t\log|dT_a|$,
and the geometric pressure function is real analytic.
At the first bifurcation parameter $a=2$
the Julia set contains the critical point, and so the dynamics is nonhyperbolic and
structurally unstable.
 The Lyapunov exponent of any ergodic measure is either $\log2$ or $\log4$, and it is $\log 4$ only 
 for the Dirac measure, denoted by $\delta_{-1}$, at the orientation-preserving fixed point.
 Equilibrium measures for the potential $-t\log|dT_2|$ are: (i) $\delta_{-1}$ if $t<-1$;
  (ii) $\delta_{-1}$ and $\mu_{\rm ac}$ if $t=-1$; (iii) $\mu_{\rm ac}$ if $t>-1$,
where $\mu_{\rm ac}$ denotes the absolutely continuous invariant probability measure. 
 Correspondingly, the pressure function is not real analytic at $t=-1$:
$$P(-t\log|dT_2|)=\begin{cases}
-t\log 4\ \ \text{\rm if} \ t\leq -1;\\
 (1-t)\log 2 \ \  \text{\rm if} \ t>-1.\end{cases}$$
 We say $T_2$ displays the {\it freezing phase transition in negative spectrum},
to be defined below (see the paragraph just before the Main Theorem).

 This phase transition is due to the fact that the measure $\delta_{-1}$ is anomalous:
 it has the maximal Lyapunov exponent, and this value is isolated in the set of Lyapunov exponents of all ergodic measures.
 For all $a\in(-1/4,2)$ the Dirac measure at the orientation preserving fixed point continues to be anomalous, 
 and therefore all the quadratic maps continue to display the
 freezing phase transition \cite[Proposition 4]{Dob09}.
  The freezing phase transition in negative spectrum is often caused by anomalous periodic points.
 For example, see \cite{Lop90} for results on certain
two-dimensional real polynomial endomorphism,
and  \cite{MakSmi00}
 for a complete characterization on
  rational maps of degree $\geq2$ on the Riemannian sphere.

An elementary observation is that
any nonhyperbolic one-dimensional map  sufficiently close to $T_2$ in the $C^2$-topology
displays the freezing phase transition in negative spectrum.
This raises the following question: is it possible to remove the phase transition of $T_2$
by an arbitrarily small singular perturbation to higher dimensional nonhyperbolic maps?
More precisely we ask:
\\

({\bf Removability problem}) {\it Is it possible to ``approximate" $T_2\colon x\in\mathbb R\mapsto 1-2x^2$ by higher dimensional nonhyperbolic maps
which do not display the freezing phase transition in negative spectrum?}\\

The aim of this paper is to show that the phase transition of $T_2$ can be removed,
by an arbitrarily small singular perturbation along {\it the first bifurcation curve} of
a family of H\'enon-like diffeomorphisms
\begin{equation*}\label{henon}
f_{a}\colon(x,y)\in\mathbb R^2\mapsto(1-ax^2,0)+b\cdot\Phi(a,b,x,y),\quad a\in\mathbb R, \ 0<b\ll1,
\end{equation*}
where $a$ is near $2$, $\Phi$ is bounded continuous in $a,b,x,y$ and 
$C^2$ in $a,x,y$. 
The parameter $a$ controls the nonlinearity, and the $b$ controls the dissipation of the map.
Note that, with $b=0$ the family degenerates into the family of quadratic maps.


We proceed to recall some known facts on the first bifurcation of the family of H\'enon-like diffeomorphisms.
If there is no fear of confusion, we suppress $a$ from notation and write $f$ for $f_a$, and so on. 
For $(a,b)$ near $(2,0)$
let $P$, $Q$ denote the fixed saddles of $f$ near $(1/2,0)$ and $(-1,0)$ respectively.
The stable and unstable manifolds of $P$ are respectively defined as follows:
$$W^s(P)=\{z\in\mathbb R^2\colon f^n(z)\to P\text{ as }n\to+\infty\};$$
$$W^u(P)=\{z\in\mathbb R^2\colon f^n(z)\to P \text{ as }n\to-\infty\}.$$
The stable and unstable manifolds of $Q$ are defined in the same way.
It is known \cite{BedSmi06,CLR08,DevNit79,Tak13} that 
there is a  \emph{first bifurcation parameter}
 $a^*=a^*(b)\in\mathbb R$ 
with the following properties:

\begin{figure}
\begin{center}
\includegraphics[height=4cm,width=16cm]
{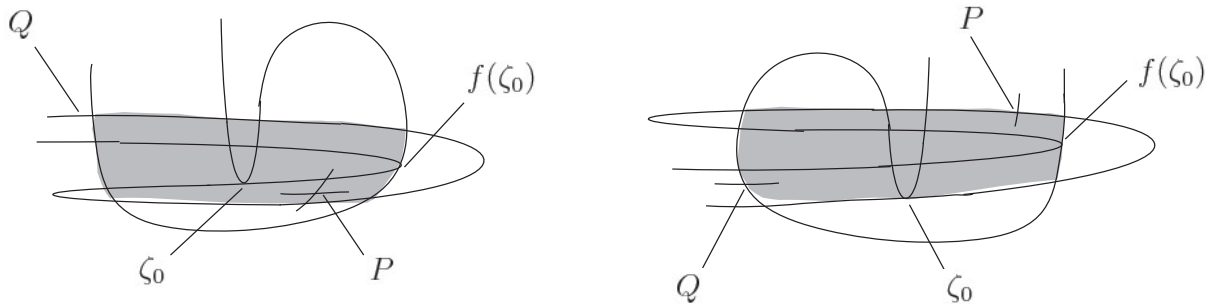}
\caption{Organization of the invariant manifolds at $a=a^*$. There exist two
fixed saddles $P$, $Q$ near $(1/2,0)$, $(-1,0)$
respectively. In the case $\det Df>0$ (left),
$W^s(Q)$ meets $W^u(Q)$ tangentially. In the case $\det Df<0$
(right), $W^s(Q)$ meets $W^u(P)$ tangentially. The shaded
regions represent the rectangle $R$ (See Sect.\ref{family}).}
\end{center}
\end{figure}

\begin{itemize}

\item if $a>a^*$, then the non wandering
set is a uniformly hyperbolic horseshoe;

\item if $a=a^*$, then there is a single orbit of homoclinic or heteroclinic tangency involving (one of) the two fixed saddles (see FIGURE 1).
In the case $\det Df>0$ (orientation preserving), $W^s(Q)$ meets $W^u(Q)$ tangentially.
In the case $\det Df<0$ (orientation reversing), $W^s(Q)$ meets $W^u(P)$ tangentially.
The tangency is quadratic, and the one-parameter family $\{f_{a}\}_{a\in\mathbb R}$ unfolds the tangency at $a^*$ generically. An incredibly rich array of dynamical complexities is unleashed
in the unfolding of this tangency (see e.g., \cite{PalTak93} and the references therein);

\item $a^*\to2$ as $b\to0$.

\end{itemize}
The curve $\{(a^*(b),b)\in\mathbb R\colon b>0\}$ is a {\it nonhyperbolic path} to the quadratic map $T_2$,
consisting of parameters corresponding to nonhyperbolic dynamics.
The main theorem claims that $f_{a^*(b),b}$ does {\it not} display the freezing phase transition 
in negative spectrum.

\begin{figure}
\begin{center}
\includegraphics[height=7cm,width=9.5cm]
{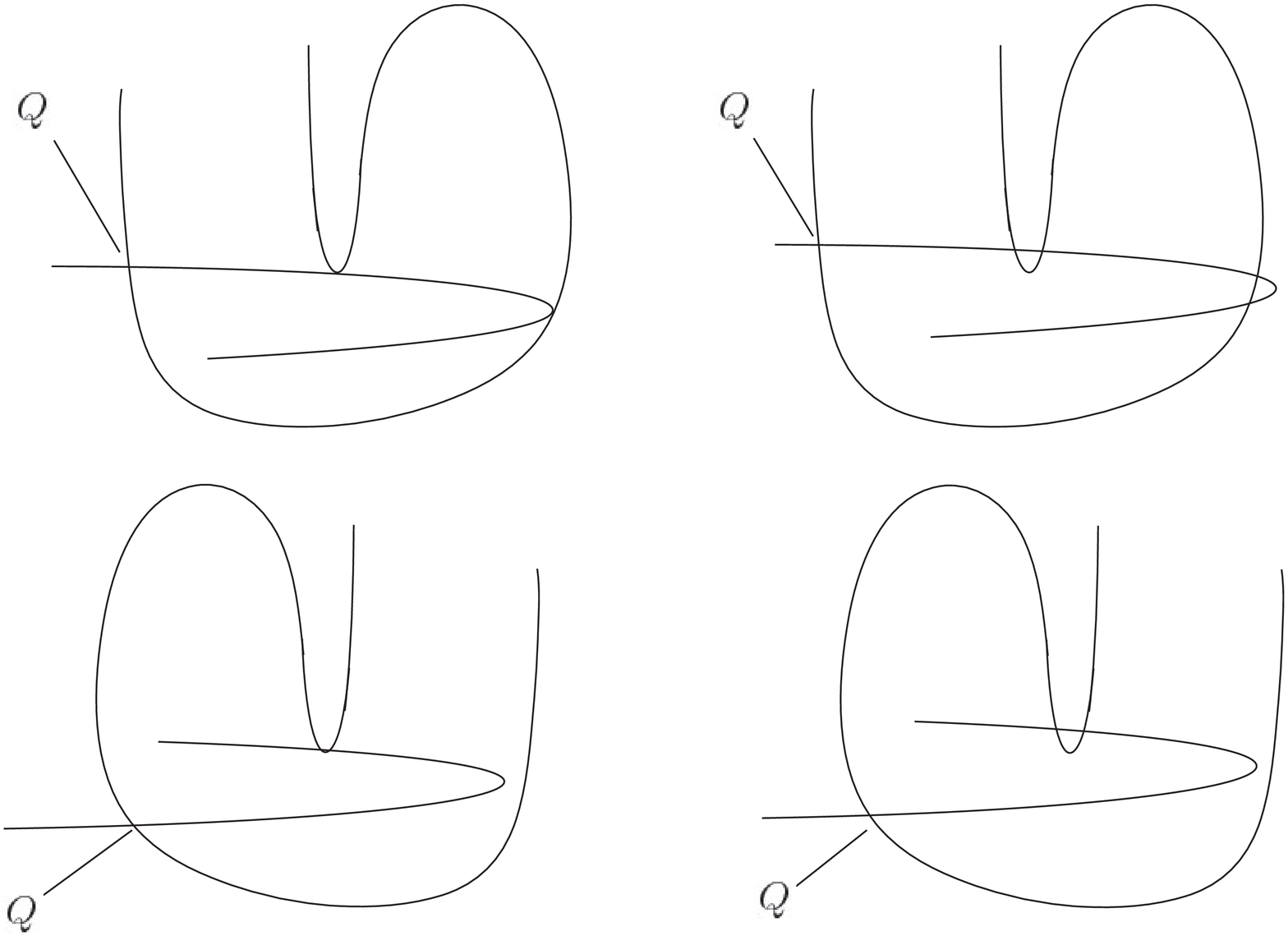}
\caption{Organization of $W^s(Q)$ and $W^u(Q)$: $\det Df>0$ and $a>a^{**}$ close to $a^{**}$ (upper-right); 
$\det Df>0$ and $a=a^{**}$ (upper-left);
$\det Df<0$ and $a>a^{**}$ close to $a^{**}$ (lower-right); $\det Df<0$ and $a=a^{**}$ (lower-left).
}
\end{center}
\end{figure}

To give a precise statement of result we need a preliminary discussion.
We first make explicit the range of the parameter $a$ to consider. Assume $0<b\ll1$.
Let $W^s_{\rm loc}(Q)$ denote the compact curve in $W^s(Q)$ containing
$Q$ such that $W^s_{\rm loc}(Q)\setminus \{Q\}$ has two connected components 
of length $\sqrt{b}$.
Let $\psi\colon\mathbb R\to W^u(Q)$ denote the isometric embedding such that $\psi(0)=Q$
and $\psi(\{x\in\mathbb R\colon x<0\})\cap\Omega=\emptyset$.
Let $$\ell^u=\begin{cases}
\psi(1-1/100,1+1/100)&\text{ if $\det Df>0$;}\\
\psi(3-1/100,3+1/100)&\text{ if $\det Df<0$.}
\end{cases}$$
 Define $$\mathcal G=\{a\in\mathbb R\colon \text{$f^{-2}(W^s_{\rm loc}(Q))$ and $\ell^u$ bounds a compact domain}\},$$
 and $$a^{**}=\inf\mathcal G.$$ 
  Note that $a^*\in\mathcal G$, $a^{**}<a^*$, $a^{**}\to2$ as $b\to0$, and that at $a=a^{**}$,
 $f^{-2}(W^s_{\rm loc}(Q))$ is tangent to $\ell^u$ quadratically.
 Since the family $\{f_a\}_a$ unfolds the tangency $\zeta_0$ at $a=a^*$ generically,
$(a^{**},a^*]\subset\mathcal G$. 
In this paper we assume $a\in(a^{**},a^*]$.

Let $\Omega$ denote 
the non wandering set of $f$, which is a compact $f$-invariant set. 
For nonhyperbolic dynamics beyond the parameter $a^*$,
the notion of ``unstable direction" is not clear.
In the next paragraph, we circumvent this point 
with the Pesin theory (See e.g., \cite{Kat80}),
by introducing a Borel set $\Lambda$
on which {\it an unstable direction} $E^u$ makes sense.

Given $\chi>\epsilon>0$, for each integer $k\geq1$
define $\Lambda_k(\chi,\epsilon)$ to be the set of points
$z\in \Omega$ for which there is a one-dimensional subspace $E_z^u$ of $T_z\mathbb R^2$ such that
for every integers $m\in\mathbb Z$, $n\geq1$ and for all vectors $v^u\in Df^m(E^u_z)$,
$$\|D_{f^m(z)}f^{-n}(v^u)\|\leq  e^{\epsilon k}e^{-(\chi-\epsilon)n}e^{\epsilon|m|}\|v^u\|.$$
Since $f^{-1}$ expands area, the subspace $E^u_z$ with this property is unique when it exists,
and characterized by the following backward contraction property
\begin{equation}\label{eu}\limsup_{n\to+\infty}\frac{1}{n}\log\|Df^{-n}|E^u_z\|< 0.\end{equation}
Here, $\|\cdot\|$ denotes the norm induced from the Euclidean metric on $\mathbb R^2$.
Note that $\Lambda_k(\chi,\epsilon)$ is a closed set, and $z\in\Lambda_k(\chi,\epsilon)
\mapsto E_z^u$ is continuous.
Moreover, if $z\in\Lambda_k(\chi,\epsilon)$ then $f(z),f^{-1}(z)\in\Lambda_{k+1}(\chi,\epsilon)$.
Therefore, the Borel set $$\Lambda(\chi,\epsilon)=\bigcup_{k=1}^\infty\Lambda_k(\chi,\epsilon)$$
is $f$-invariant: $f(\Lambda(\chi,\epsilon))=\Lambda(\chi,\epsilon)$.
Then the Borel set
$$\Lambda=\bigcup_{\epsilon>0}\bigcup_{\chi>\epsilon}\Lambda(\chi,\epsilon)$$
is $f$-invariant as well, and the map $z\in\Lambda\mapsto E_z^u$ is Borel measurable with the invariance property 
 $Df(E^u_z)=E^u_{f(z)}$.
 The one-parameter family of potentials we are concerned with is 
$$-t\log J^u\quad t\in\mathbb R,$$
where $J^u(z)=\|Df|E^u_z\|$ $(z\in\Lambda)$.
Since $\Omega$ is compact and $f$ is a diffeomorphism, 
$J^u$ is bounded from above
and bounded away from zero.
We shall only take into consideration measures which give full weight to $\Lambda$.

\begin{figure}
\begin{center}
\includegraphics[height=4.5cm,width=7cm]
{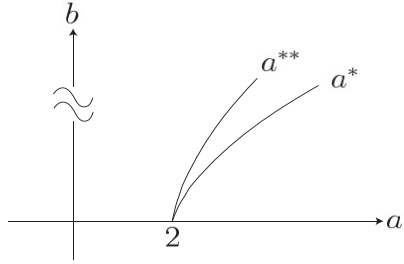}
\caption{The landscape in $(a,b)$-space, $b\ll1$. The parameters $a^*=a^*(b)$ and $a^{**}=a^{**}(b)$ converge to $2$ as $b\to0$.
The dynamics for parameters at the right of the $a^*$-curve is uniformly hyperbolic.}
\end{center}
\end{figure}

The chaotic behavior of $f$ is produced by the non-uniform
expansion along the unstable direction $E^u$, and thus
a good deal of information will be obtained by studying 
the associated geometric pressure function $t\in\mathbb R\mapsto P(-t\log J^u)$ defined by
\begin{equation}\label{pressure2}P(-t\log J^u)=\sup\left\{h_\mu(f)-t\lambda^u(\mu)\colon
\mu\in\mathcal M_0(f)\right\},\end{equation}
where $$\mathcal M_0(f)=\{\mu\in\mathcal M(f)\colon\mu(\Lambda)=1\},$$
and  $h_\mu(f)$ denotes the entropy of $(f,\mu)$,
and $$\lambda^u(\mu)=\int\log J^ud\mu,$$
which we call an {\it unstable Lyapunov exponent} of $\mu\in\mathcal M_0(f)$.
Let us call  any measure in $\mathcal M(f)$ which attains the supremum in $P(-t\log J^u)$ 
{\it an equilibrium
measure} for the potential $-t\log J^u$.

We suggest the reader to compare \eqref{pressure1} and \eqref{pressure2}.
One important difference is that the function $\log|dT_a|$ in \eqref{pressure1} is unbounded,
while the function $\log J^u$ in \eqref{pressure2} is bounded.
Another important difference is that the class of measures taken into consideration is reduced
in \eqref{pressure2}.

It is natural to ask in which case $\mathcal M_0(f)=\mathcal M(f)$.
This is the case for  $a=a^*$ because $\Lambda=\Omega$ from the result in
  \cite{SenTak1}.
In fact, $\mathcal M_0(f)=\mathcal M(f)$ still holds 
for ``most" parameters immediately right after the first bifurcation at $a^*$.
See Sect.\ref{equal}  for more details.

The potential $-t\log J^u$ and the associated pressure function 
deserve to be called {\it geometric}, primarily because
 Bowen's formula holds at $a=a^*$ \cite[Theorem B]{SenTak2}: 
 the equation $P(-t\log J^u)=0$ has the unique solution which coincides
with the (unstable) Hausdorff dimension of $\Omega$.
We expect that the same formula holds for the above ``most" parameters.

We are in position to state our main result.
Let 
\begin{align*}
\lambda_m^u&=\inf\{\lambda^u(\mu)\colon\mu\in\mathcal M_0(f)\};\\
\lambda_M^u&=\sup\{\lambda^u(\mu)\colon\mu\in\mathcal M_0(f)\},
\end{align*}
and define {\it freezing points} $t_c$, $t_f$ by
\begin{align*}
t_c&=\inf\left\{t\in\mathbb R\colon P(-t\log J^u)>-t\lambda_M^u\right\};\\
t_f&=\sup\{t\in\mathbb R\colon P(-t\log J^u)>-t\lambda_m^u\}.
\end{align*}
This denomination is because equilibrium measures do not change any more for $t<t_c$
or $t>t_f$.
Indeed it is elementary to show the following:
\begin{itemize}
\item $-\infty\leq t_c<0<t_f\leq+\infty$;

\item if $t\in(t_c,t_f)$, then
any equilibrium measure for $-t\log J^u$ (if it exists) has positive entropy;

\item if $t\leq t_c$, then $P(-t\log J^u)=-t\lambda_M^u$.
If $t\geq t_f$, then $P(-t\log J^u)=-t\lambda_m^u$.

\end{itemize}

Let us say that {\it $f$ displays the freezing phase transition in negative (resp. positive) spectrum}
if $t_c$ (resp. $t_f$) is finite.


\begin{figure}
\begin{center}
\includegraphics[height=4cm,width=6.5cm]
{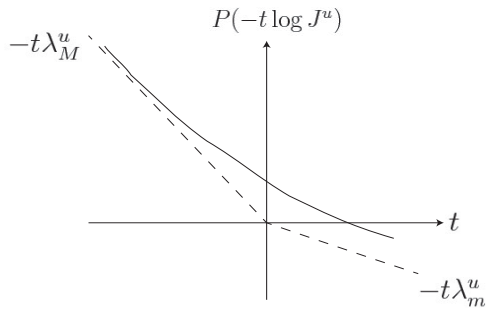}
\caption{At $a=a^*$, 
the graph of the pressure function $t\mapsto P(-t\log J^u)$ has the line $-t\lambda_M^u$ as its asymptote as $t\to-\infty$, but never touches it (Main Theorem).}
\end{center}
\end{figure}

\begin{theorema}
Let $\{f_{a}\}$ be a family of H\'enon-like diffeomorphisms. 
If $b>0$ is sufficiently small 
and $a\in(a^{**}(b),a^*(b)]$, 
then $f_{a}$ does not display the freezing phase transition in negative spectrum.
If $a=a^*(b)$, then $P(-t\log J^u)=-t\lambda_M^u+o(1)$ as $t\to-\infty$. 
\end{theorema}
The main theorem states that
 the graph of the pressure function does not touch the line $-t\lambda_M^u$.
At $a=a^*$ we have more information: this line is the asymptote of the graph of 
$P(-t\log J^u)$ as $t\to-\infty$ (see FIGURE 4).

The main theorem reveals a difference between the bifurcation structure of quadratic maps
and that of H\'enon-like maps from the thermodynamic point of view. As mentioned earlier, 
the quadratic maps display the freezing phase transition in negative spectrum
for all parameters beyond the bifurcation, while this is not the case for H\'enon-like
maps. 

The freezing phase transition for negative spectrum does occur for 
some parameters $<a^{**}$.
It is well-known that there exists a parameter set of 
 positive Lebesgue measure corresponding
to non-uniformly hyperbolic strange attractors \cite{BenCar91,MorVia93,WanYou01}. 
For these parameters, the non-wandering set is the disjoint union of the strange attractor 
and the fixed saddle near $(-1,0)$
\cite{Cao99, CaoMao00}. For these parameters it is possible to show that the Dirac measure at 
the saddle is anomalous.

Regarding freezing phase transitions in positive spectrum of H\'enon-like maps,
the known result is very much limited. Let $\delta_Q$ denote the Dirac measure at $Q$.
It was proved in \cite[Proposition 3.5(b)]{Tak15} that if $a=a^*$ and
$\lambda_m^u=(1/2)\lambda^u(\delta_Q)$, then $f$ does not display the freezing phase transition in positive spectrum. However, since $\lambda_m^u\to\log2$ and $\lambda^u(\delta_Q)\to\log4$ as $b\to0$, 
it is not easy to prove or disprove this equality.

For a proof of the main theorem we first show that 
 $\delta_Q$ is the unique measure
which maximizes the unstable Lyapunov exponent (see Lemma \ref{maximal}). 
Then it suffices to show that for any $t<0$ there exists a measure $\nu_t\in\mathcal M_0(f)$ such that
$h_{\nu_t}(f)-t\lambda^u(\nu_t)>-t\lambda^u(\delta_Q).$
To see the subtlety of showing this, note that 
from the variational principle $\nu_t$ must satisfy
\begin{equation}\label{require}
t\left(\lambda^u(\nu_t)-\lambda^u(\delta_Q)\right)<h_{\nu_t}(f)\leq h_{\rm top}(f),\end{equation}
where $h_{\rm top}(f)$ denotes the topological entropy of $f$.
As $-t$ becomes large, the unstable Lyapunov exponent becomes more important and
we must have
$\lambda^u(\nu_t)\to\lambda^u(\delta_Q)$ as $t\to-\infty$.
A naive application of the Poincar\'e-Birkhoff-Smale theorem \cite{PalTak93}
to a transverse homoclinic point of $Q$ indeed yields
a measure whose unstable Lyapunov 
exponent is approximately that of $\delta_Q$, but it is not clear if the entropy
is sufficiently large for the first inequality in \eqref{require} to hold.

Our approach is based on the well-known inducing techniques adapted
to the H\'enon-like maps, inspired by Makarov $\&$ Smirnov 
\cite{MakSmi00} (see also Leplaideur \cite{Lep11}). 
The idea is to carefully choose for each $t<0$ a hyperbolic subset $H_t$ of $\Omega$ such that
the first return map to it is topologically conjugate to the full shift on a finite number of symbols.
We then spread out the maximal entropy measure of the first return map to produce a measure 
with the desired properties. As $-t$ becomes large, more symbols are needed
in order to fulfill the first inequality in \eqref{require}.

The hyperbolic set $H_t$ is chosen in such a way that any orbit contained in it spends 
a very large proportion of time near the saddle $Q$, during which the unstable directions are roughly parallel
to $E^u_Q$. More precisely, for any point $z\in H_t$
with the first return time $R(z)$ to $H_t$, the fraction
$$\frac{1}{R(z)}\#\{n\in\{0,1,\ldots,R(z)-1\}\colon \text{$|f^n(z)-Q|\ll1$ and 
${\rm angle}(E^u_{f^n(z)},E^u_Q)\ll1$}\}$$
is nearly $1$.
A standard bounded distortion argument then allows us to copy the unstable Lyapunov exponent of $\delta_Q$.
Note that, if the unstable direction is not continuous (which is indeed the case at $a=a^*$ \cite{SenTak2} and considered to be the case for most $a<a^*$),
then the closeness of base points $|f^n(z)-Q|\ll1$  does not guarantee 
the closeness of the corresponding unstable directions ${\rm angle}(E^u_{f^n(z)},E^u_Q)\ll1$.

In order to let points stay near the saddle $Q$ for a very long period of time,
one must allow them to enter deeply into the critical zone.
As a price to pay, the directions of $E^u$ along the orbits get switched due to the folding behavior near the critical zone.
In order to restore the horizontality of the direction and establish the closeness to $E^u_Q$,
we develop the binding argument relative to dynamically critical points, inspired by Benedicks $\&$ Carleson \cite{BenCar91}. 
The point is that one can choose the hyperbolic set $H_t$ so that the effect of the folding is not significant,
 and the restoration can be done in a uniformly bounded time.
 This argument works at the first bifurcation parameter $a^*$, and even for all parameters
in $(a^{**},a^*)$ because only those parts in the phase space not being destroyed by the homoclinic bifurcation 
are involved.



The rest of this paper consists of three sections. In Sect.2 we introduce the key concept of critical points,
and develop estimates related to them.
 In Sect.3 we use the results in Sect.2 to construct the above-mentioned hyperbolic set.
In Sect.4 we  finish the proof of the main theorem and provide more details, 
on the abundance of parameters satisfying $\mathcal M_0(f)=\mathcal M(f)$.

\section{Local analysis near critical orbits}
For the rest of this paper, we assume $f=f_a$ and $a\in(a^{**},a^*]$.
In this section we develop a local analysis near the orbits of critical points.
 The main result is Proposition \ref{binding} which controls the norms of the derivatives in the unstable direction,
 along orbits which pass through critical points.

For the rest of this paper 
we are concerned with the following positive small constants: $\tau$, $\delta$, $b$ chosen in this order,
the purposes of which are as follows:
\begin{itemize}

\item $\tau$ is used to exclusively in the proof of Proposition \ref{binding};

\item $\delta$ determines the size of a critical region (See Sect.\ref{critical});

\item $b$ determines the magnitude of the reminder term in \eqref{henon}.
\end{itemize}
We shall write $C$ with or without indices to denote any constant which is independent of $\tau$, $\delta$, $b$.
For $A,B>0$ we write $A\approx B$ if both $A/B$ and $B/A$ are bounded from above
by constants independent of $\tau$, $\delta$, $b$.
If $A\approx B$ and the constants can be made arbitrarily close to $1$ by appropriately choosing $\tau$, $\delta$, $b$,
then we write $A\asymp B$.

For a nonzero tangent vector $v=\left(\begin{smallmatrix}\xi\\\eta\end{smallmatrix}\right)$ at a point $z\in\mathbb R^2$, define
${\rm slope}(v)=|\eta|/|\xi|$ if $\xi\neq0$, and
${\rm slope}(v)=\infty$ if $\xi=0$.
Similarly, for the one-dimensional subspace $V$ of $T_z\mathbb R^2$ spanned by $v$, define
${\rm slope}(V)={\rm slope}(v)$.
Given a $C^1$ curve $\gamma$ in $\mathbb R^2$, the length is denoted by ${\rm length}(\gamma)$.
The tangent space of $\gamma$ at $z\in\gamma$ is denoted by $T_z\gamma$.
The Euclidean distance between two points $z_1,z_2$ of $\mathbb R^2$ is denoted by $|z_1-z_2|$.
The angle between two tangent vectors $v_1$, $v_2$ is denoted by ${\rm angle}(v_1,v_2)$.
The interior of a subset $X$ of $\mathbb R^2$ is denoted by ${\rm Int}(X)$.

\subsection{The non wandering set}\label{family}
Recall that the map $f$ has exactly two fixed points, which are saddles:
$P$ is the one near $(1/2,0)$ and $Q$ is the other one near $(-1,0)$.
The orbit of tangency at the first bifurcation parameter $a=a^*$ 
intersects a small neighborhood of the origin $(0,0)$ exactly at one point, denoted by $\zeta_0$.
If $\det Df>0$ then $\zeta_0\in W^s(Q)\cap W^u(Q)$.
If $\det Df<0$ then $\zeta_0\in W^s(Q)\cap W^u(P)$
(See FIGURE 1).

By a {\it rectangle} we mean any
compact domain bordered by two compact curves in $W^u(P)\cup W^u(Q)$ and two in 
$W^s(P)\cup W^s(Q)$. By an {\it unstable side} of a
rectangle we mean any of the two boundary curves in $W^u(P)\cup W^u(Q)$. A {\it
stable side} is defined similarly.

In the case $\det Df>0$ (resp. $\det Df<0$) let $R=R_{a}$ denote the rectangle which is
 bordered by two compact curves in $W^u(Q)$ (resp. $W^u(P)$) and two in $W^s(Q)$,
and contains $\Omega$.
The rectangle with these properties is unique, and is
 located near the segment $\{(x,0)\in\mathbb R^2\colon |x|\leq1\}$.
 One of the stable sides of $R$ contains $Q$, which is denoted by $\alpha_0^-$.
The other stable side of $R$ is denoted by $\alpha_0^+$.
We have $f(\alpha_0^+)\subset\alpha_0^-$.
At $a=a^*$, one of the unstable sides of $R$ contains the point $\zeta_0$
of tangency near $(0,0)$
 (See FIGURE 1).

\subsection{Non critical behavior}
Define
$$I(\delta)=\{(x,y)\in R\colon |x|<\delta\},$$
and call it {\it a critical region}.

By a \emph{$C^2(b)$-curve} we mean a compact, nearly horizontal $C^2$ curve in $R$ such that the slopes of tangent vectors to it 
are $\leq\sqrt{b}$ and the curvature is everywhere $\leq\sqrt{b}$. 

\begin{lemma}
Let $\gamma$ be a $C^2(b)$-curve in $R\setminus I(\delta)$. 
Then $f(\gamma)$ is a $C^2(b)$-curve.
\end{lemma}
\begin{proof}
Follows from Lemma \ref{expand} and the lemma below.
\end{proof}

Put
$\lambda_0=\frac{99}{100}\log2.$
\begin{lemma}\label{expand}
Let $z\in R$ and $n\geq1$ be an integer such that $z,f(z),\ldots,f^{n-1}(z)\notin I(\delta)$.
Then for any nonzero vector $v$ at $z$ with ${\rm slope}(v)\leq\sqrt{b}$,
$${\rm slope}(Df^n(v))\leq\sqrt{b}\ \text{ and }\ \|Df^n(v)\|\geq \delta e^{\lambda_0 n}\|v\|.$$
If moreover $f^n(z)\in I(\delta)$, then 
$\|Df^n(v)\|\geq e^{\lambda_0 n}\|v\|.$
\end{lemma}
\begin{proof}
Follows from the fact that $|DT_2(x)|>2$ outside of $[-1,1]$ and that $b$ is very small.
\end{proof}
\begin{lemma}\label{curvature}{\rm (\cite[Lemma 2.3]{Tak11})}
Let $\gamma$ be a $C^2$ curve in $R$ and $z\in\gamma$. For each integer $i\geq0$ let $\kappa_i(z)$ denote the curvature of $f^i(\gamma)$ at $f^i(z)$.
Then for any nonzero vector $v$ tangent to $\gamma$ at $z$,
$$\kappa_i(z)\leq (Cb)^i\frac{\|v\|^3}{\|Df^i(v)\|^3}\kappa_0(z)+\sum_{j=1}^i(Cb)^j\frac{\|Df^j(v)\|^3}{\|Df^i(v)\|^3}.$$
\end{lemma}

\subsection{Lyapunov maximizing measure}\label{maximizing}
Recall that $\mathcal M_0(f)$ is the set of $f$-invariant Borel probability measures which gives total mass to the set $\Lambda$.
The next lemma states that $\delta_Q$ is the unique measure which maximizes the unstable Lyapunov exponent
among measures in $\mathcal M_0(f)$.
\begin{lemma}\label{maximal}
For any $\mu\in\mathcal M_0(f)\setminus\{\delta_Q\}$, $\lambda^u(\mu)<\lambda^u(\delta_Q)$.
\end{lemma}
\begin{proof}
From the linearity of the unstable Lyapunov exponent as a function of measures,
it suffices to consider the case where $\mu$ is ergodic.
Let $${\rm supp}(\mu)=\bigcap\{F\colon \text{$F$ is a closed subset of $\Omega$ and $\mu(F)=1$}\}.$$
Reducing $b>0$ if necessary,
one can show that $\lambda^u(\mu)<\lambda^u(\delta_Q)$ holds for any ergodic $\mu$
with ${\rm supp}(\mu)\cap I(\delta)=\emptyset$.
In the case ${\rm supp(\mu)}\cap I(\delta)\neq\emptyset$ we have $\mu(I(\delta))>0$. From the Ergodic Theorem, 
it is possible to take a point $z\in I(\delta)$ such that $$\displaystyle{\lim_{n\to\infty}}\frac{1}{n}\displaystyle{\sum_{i=0}^{n-1}}
\log \|Df|E^u_{f^i(z)}\|=\lambda^u(\mu)>0,$$
and
$$\lim_{n\to\infty}\frac{1}{n}\#\{i\in\{0,1,\ldots,n-1\}\colon f^i(z)\in I(\delta)\}=\mu(I(\delta))>0.$$
Define a sequence $ m_1\leq m_1+r_1\leq m_2\leq m_2+r_2\leq m_3\leq\cdots$ of nonnegative integers inductively as follows.
Start with $m_1=0$.
Let $k\geq1$ and $m_k$ be such that $f^{m_k}(z)\in I(\delta)$.
If ${\rm slope}(E^u_{f^{m_k+1}(z)})\leq 1/b^{1/3}$, then
define $$r_k=0\ \ \text{and} \ \ m_{k+1}=\min\{n>m_k\colon f^n(z)\in I(\delta)\}.$$
If ${\rm slope}(E^u_{f^{m_k+1}(z)})> 1/b^{1/3}$, then
define $$r_{k}=\min\{i>1\colon {\rm slope}(E^u_{f^{m_k+i}(z)})\leq 1/b^{1/3}\}\ \ \text{and} \ \ 
m_{k+1}=\min\{n\geq m_k+r_k\colon f^n(z)\in I(\delta)\}.$$
Since $\lambda^u(\mu)>0$, $r_k<\infty$.

The form of our map \eqref{henon} gives ${\rm slope}(E^u_{f^n(z)})\leq 1/b^{1/3}$ for every $n\in\{m_k+r_k,m_k+r_k+1,\ldots,m_{k+1}\}$.
This implies
$$\|Df^{m_{k+1}-m_k-r_k}|E^u_{f^{m_k+r_k}(z)}\|\leq 2e^{\lambda^u(\delta_Q)(m_{k+1}-m_k-r_k)}.$$
If ${\rm slope}(E^u_{f^n(z)})\geq 1/b^{1/3}$ then  $\|D_{f^n(z)}f|E^u_{f^n(z)}\|\leq \sqrt{b}$.
Hence
\begin{equation}\label{equa1}
\begin{split}\|Df^{m_{k+1}-m_k}|E^u_{f^{m_k}(z)}\| &\leq 2e^{\lambda^u(\delta_Q)(m_{k+1}-m_k-r_k)}\min\{b^{\frac{r_k}{2}},3\delta\}\\
& \leq e^{\lambda^u(\delta_Q)(m_{k+1}-m_k)}\min\{b^{\frac{r_k}{3}},3\delta\}.\end{split}
\end{equation}

For each integer $n\geq0$ define
$V_n=\{i\in\{0,1,\ldots,n-1\} \colon {\rm slope}(E^u_{f^i(z)})\geq 1/b^{1/3}\}.$
If $\displaystyle{\limsup_{n\to\infty}\#V_n/n>0}$, then
the first alternative in \eqref{equa1} yields
\begin{align*}\frac{1}{m_{k+1}}\log \|Df^{m_{k+1}}|E^u_{z}\|&\leq \lambda^u(\delta_Q)+\frac{1}{3}\log b\cdot
\frac{1}{m_{k+1}}\sum_{i=1}^kr_i\\
&= \lambda^u(\delta_Q)+\frac{1}{3}\log b\cdot\frac{1}{m_{k+1}}\#V_{m_{k+1}}.\end{align*}
Taking the upper limit as $k\to\infty$ yields
$\lambda^u(\mu)<\lambda^u(\delta_Q)$.
If $\displaystyle{\limsup_{n\to\infty}\#V_n/n=0}$, then
note that $k=\{n\in\{0,1,\ldots,m_{k+1}-1\}\colon f^n(z)\in I(\delta)\text{ and }{\rm slope}(E^u_{f^n(z)})\leq 1/b^{\frac{1}{3}}\}$
and
$k/m_{k+1}\to\mu(I(\delta))>0$ as $k\to\infty$.
Then the second alternative in \eqref{equa1} yields
\begin{align*}\frac{1}{m_{k+1}}\log \|Df^{m_{k+1}}|E^u_{z}\|&\leq \lambda^u(\delta_Q)+\log \delta\cdot \frac{k}{m_{k+1}}.\end{align*}
Taking the upper limit as $k\to\infty$ yields
$\lambda^u(\mu)<\lambda^u(\delta_Q)$.
\end{proof}

\subsection{$C^1$-closeness due to disjointness}\label{closeness}
 Corollary \ref{c2cor} below states that
the pointwise convergence of pairwise disjoint $C^2(b)$-curves implies
the $C^1$-convergence. This fact was already used in the precious works 
for H\'enon-like maps, e.g., \cite{SenTak1,SenTak2}.
We include precise statements and proofs for the reader's convenience.

\begin{lemma}\label{c2b}
Let $\varepsilon\in(0,(1+\sqrt{b})^{-2})$ and
let $\gamma_1$, $\gamma_2$ be two disjoint $C^2(b)$-curves parametrized by arc length such that:

\begin{itemize}
\item[(i)] $\gamma_1(s)$, $\gamma_2(s)$ are defined for $s\in[-\varepsilon,\varepsilon]$;
\item[(ii)] $|\gamma_1(0)-\gamma_2(0)|\leq \varepsilon^{2}.$
\end{itemize}

Then the following holds:
\begin{itemize}
\item[(a)] ${\rm angle}(T_{\gamma_1(0)}\gamma_1,T_{\gamma_2(0)}\gamma_2)\leq \sqrt{\varepsilon}$;
\item[(b)] $|\gamma_1(s)-\gamma_2(s)|\leq 2\varepsilon^{\frac{3}{2}}$ for all $s\in[-\varepsilon,\varepsilon]$.
\end{itemize}
\end{lemma}
\begin{proof}
Write $L(s)=\gamma_1(s)-\gamma_2(s)$. 
By the mean value theorem, for any $t$ in between $0$ and $s$ there exists 
 $\theta(t)$ in between $0$ and $t$ such that
$\dot {L}(t)=\dot{L}(0)+\ddot{L}(\theta(t))t$, 
where the dot $``\cdot"$ denotes the $t$-derivative.
Integrating this equality gives
\begin{equation}\label{c2}L(s)=L(0)+
\int_0^s\dot {L}(t)dt=L(0)+\dot L(0)s+\int_0^s\ddot{L}(\theta(t))tdt.\end{equation}
We argue by contradiction assuming
 $\|\dot{L}(0)\|> (1/2)\sqrt{\varepsilon}$. The assumption $\varepsilon<(1+\sqrt{b})^{-2}$, (ii) and $|\ddot{L}|\leq 2\sqrt{b}$ give
 \begin{equation}\label{c3}\left|L(0)+\int_0^{\varepsilon}\ddot{L}(\theta(t))tdt\right|
 \leq|L(0)|+2\sqrt{b}\varepsilon^2\leq
   \varepsilon^{2}+2\sqrt{b}\varepsilon^{2}<\|\dot{L}(0)\|
    \varepsilon.\end{equation}
A  comparison of \eqref{c2} with \eqref{c3} 
 shows that the sign of $L(\varepsilon)$ coincides with that of $\|\dot L(0)\|\varepsilon$.
The same argument shows that the sign of $L(-\varepsilon)$ coincides with that of $-\|\dot L(0)\|\varepsilon$.
From the intermediate value theorem it follows that $L(s)=0$ for some $s$, namely
$\gamma_1$ intersects $\gamma_2$, a contradiction.
Hence  $     {\rm angle}(T_{\gamma_1(0)}\gamma_1,T_{\gamma_2(0)}\gamma_2)          <2\|\dot{L}(0)\|\leq \sqrt{\varepsilon}$
and (a) holds.
(b) follows from \eqref{c2}, (ii), (a) and  $|\ddot{L}|\leq 2\sqrt{b}$.
\end{proof}
\begin{cor}\label{c2cor}
Let  $\{\gamma_n\}_{n=0}^{+\infty}$ be a sequence of pairwise disjoint $C^2(b)$-curves
which as a sequence of $C^2$ functions converges pointwise to a function $\gamma$ as $n\to+\infty$. Then
the graph of
$\gamma$ is a $C^1$-curve and the slopes of its tangent directions are everywhere $\leq\sqrt{b}$.
\end{cor}
\begin{proof}
From Lemma \ref{c2b}(b),
the pointwise convergence implies the uniform $C^0$ convergence. 
From Lemma \ref{c2b}(a), the uniform $C^1$ convergence follows.
\end{proof}

\begin{figure}
\begin{center}
\includegraphics[height=4cm,width=12cm]
{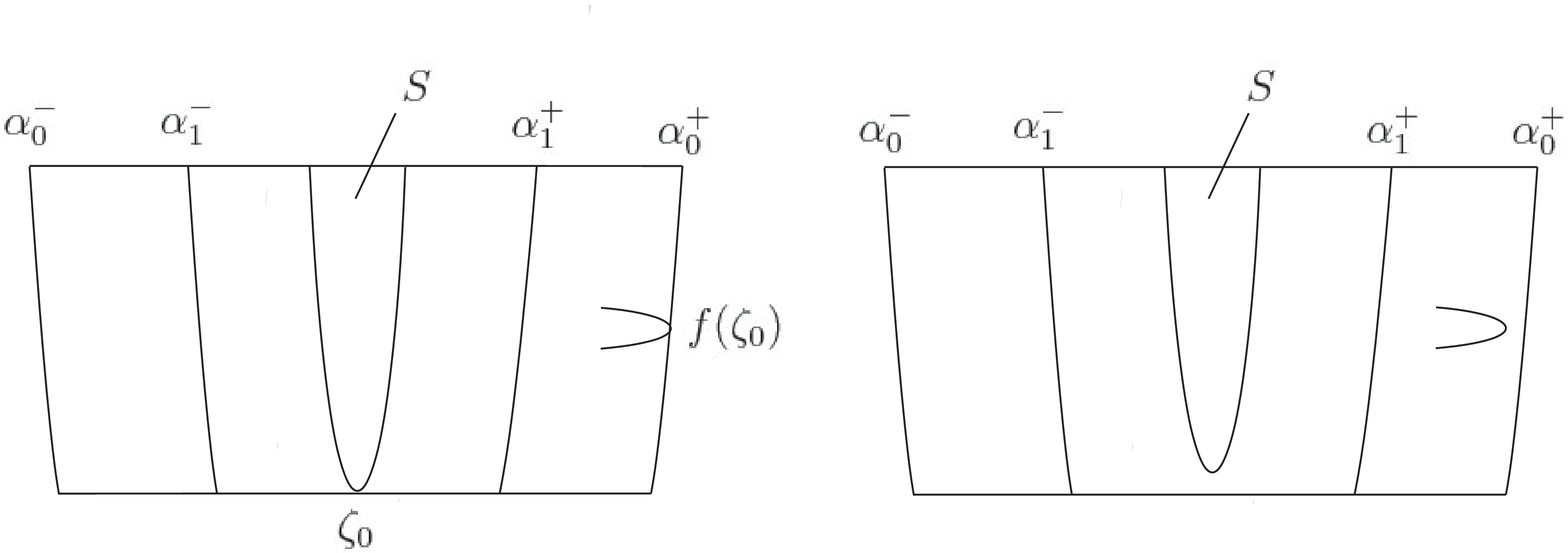}
\caption{The lenticular domain $S$: $a=a^*$ (left); $a<a^*$ (right).}
\end{center}
\end{figure}

\subsection{Critical points}\label{critical}
From the hyperbolicity of the fixed saddle $Q$,
there exist mutually disjoint connected open sets $U^-$, $U^+$ independent of $b$ such that
$\alpha_0^-\subset U^-$, $\alpha_0^+ \subset U^+$, $U^+\cap f(U^+)=\emptyset=U^+\cap f(U^-)$ and 
a foliation $\mathcal F^s$ of $U=U^-\cup U^+$ by one-dimensional vertical leaves such
that:
\begin{itemize}
\item[(a)] $\mathcal F^s(Q)$, the leaf of $\mathcal F^s$ containing $Q$,
contains $\alpha_0^-$; 
\item[(b)] if $z,f(z)\in U$, then $f(\mathcal F^s(z))
\subset\mathcal F^s(f(z))$;

\item[(c)] let $e^s(z)$ denote the unit vector in $T_z\mathcal F^s(z)$ whose second component is positive. 
Then $z\mapsto e^s(z)$ is $C^{1}$, $\|D_zfe^s(z)\|\leq Cb$ and $\|D_ze^s(z)\|\leq C$;

\item[(d)] if $z,f(z)\in U$, then 
${\rm slope}(e^s(z))\geq
C/\sqrt{b}$.
\end{itemize}

{\rm Let $\gamma$ be a $C^2(b)$-curve in $I(\delta)$.
We say $\zeta\in\gamma$ is a {\it critical point on} $\gamma$ if
$\zeta\in S$ and
$f(\gamma)$ is tangent to $\mathcal F^s(f(\zeta))$.}
If $\zeta$ is a critical point on a $C^2(b)$-curve $\gamma$, then
we say $\gamma$ {\it admits} $\zeta$. For simplicity, we sometimes refer to $\zeta$ as a {\it critical point} without referring to $\gamma$. 

Let $\zeta$ be a critical point. 
Note that $f(\zeta)\in U^+$, and the forward orbit of $f(\zeta)$ spends a long time in $U^-$. Hence
it inherits the exponential growth of derivatives near the fixed saddle $Q$.
For an integer $i\geq1$ write
$w_i(\zeta)=D_{f(\zeta)}f^{i-1}\left(\begin{smallmatrix}1\\0\end{smallmatrix}\right)$, and
 define
$$n(\zeta)=\sup\left(\{i>1\colon f^i(\zeta)\in U^-\}\cup\{+\infty\}\right).$$
Then
\begin{equation}\label{eq-1}
3^{i-1}<\|w_i(\zeta)\|<5^{i-1}\ \ \text{for every }i\in\{1,2,\ldots, n(\zeta)\}.
\end{equation}
More precisely, from the bounded distortion near the fixed saddle $Q$,
\begin{equation}\label{eq-2}
\|w_i(\zeta)\|\asymp e^{\lambda^u(\delta_Q)(i-1)}\ \ \text{for every }i\in\{1,2,\ldots, n(\zeta)\}.
\end{equation}

 \subsection{Binding to critical points}\label{bindarg}
 In order to deal with the effect of returns to $I(\delta)$,
 we now establish a binding argument in the spirit of Benedicks $\&$ Carleson \cite{BenCar91} which allows one to bind generic orbits which fall inside $I(\delta)$
to suitable critical points, to let it copy the exponential growth along the piece of the critical orbit.

Let $\zeta$ be a critical point and let
 $z\in I(\delta)\setminus\{\zeta\}$.
We define a {\it bound period} $p=p(\zeta,z)$ in the following manner.
Consider the leaf $\mathcal F^s(f(\zeta))$ of the stable foliation through $f(\zeta)$.
This leaf is expressed as a graph of a $C^2$ function: there exists an open interval
$J$ containing $0$ and independent of $b$, and a $C^2$ function $y\mapsto x^s(y)$ on $J$
such that
$$\mathcal F^s(f(\zeta))=\{(x^s(y),y)\colon y\in J\}.$$
Choose a small number $\tau>0$ such that any closed ball of radius $\sqrt{\tau}$ about a point in $\alpha_0^-$ is contained in $U^-$.
For each integer $k\geq1$ define
$$D_k(\zeta)=\tau\left[\sum_{i=1}^k\frac{\|w_{i}(\zeta)\|^2}{\|w_{i+1}(\zeta)\|}\right]^{-1}.$$
Write $f(z)=(x_0,y_0)$.
If  $|x_0-x^s(y_0)|\leq D_{n(\zeta)}(\zeta)$, then define $p=n(\zeta)+1$.
Otherwise, define $p$ to be the unique integer in $\{2,3,\ldots,n(\zeta)\}$ that satisfies 
$D_p(\zeta)<|x_0-x^s(y_0)|\leq D_{p-1}(\zeta)$.

Let $S$ denote the compact lenticular domain bounded by the parabola in $W^s(Q)$ and one of the unstable 
sides of $R$ (See FIGURE 5). 
Note that $f(S)\subset U^+$.

 \begin{prop}\label{binding}
 Let $\gamma$ be a $C^2(b)$-curve in $I(\delta)$ and $\zeta$ a critical point on $\gamma$.
 Let $z\in \gamma\setminus\{\zeta\}$ and $p=p(\zeta,z)$.
  Then the following holds.

\begin{itemize}

\item[(I)] If $p\leq n(\zeta)$, then:

\begin{itemize}

\item[(a)] $p\approx-\log|\zeta-z|$;

\item[(b)] $f^i(z)\in U$ for every $i\in\{1,2,\ldots,p-1\}$;


\item[(c)] let $v$ denote any nonzero vector tangent to $\gamma$ at $z$. Then 
$\|D_zf^p(v)\|\geq e^{\lambda_0(p-i)} \|D_zf^i(v)\|$ for every $i\in\{0,1,\ldots,p-1\}$.
In particular, ${\rm slope}(D_zf^p(v))\leq\sqrt{b}$;

\end{itemize}

\item[(II)] If $p=n(\zeta)+1$, then $f^{n(\zeta)}(z)\notin R$.

\end{itemize}
\end{prop}

A proof of Proposition \ref{binding} is lengthy. Before entering it we give a couple of remarks and prove one lemma which will be used later.

\begin{remark}\label{admissible}
{\rm Let $z\in I(\delta)\cap\Lambda$ and suppose that ${\rm slope}(E_z^u)\leq\sqrt{b}$. 
We claim that if there exists a $C^2(b)$-curve which is tangent to $E_z^u$ and 
contains a critical point $\zeta$, then $p(\zeta,z)\leq n(\zeta)$ holds.
For otherwise $p(\zeta,z)=n(\zeta)+1$, and Proposition \ref{binding}(II) gives $f^{n(\zeta)}(z)\notin R$.
Since $z\in\Lambda\subset\Omega\subset R$ a contradiction arises.}
\end{remark}

\begin{remark}\label{admit}
{\rm As a by-product of the proof of Proposition \ref{binding} it follows that
any $C^2(b)$-curve in $I(\delta)$ 
admits at most one critical point.} 
\end{remark}

Let $\alpha_1^+$ denote the connected component of $R\cap W^s(P)$ containing $P$, and
 $\alpha_1^-$ the connected component of $W^s(P)\cap f^{-1}(\alpha_1^+)$ which is not $\alpha_1^+$.
Let $\Theta$ denote the rectangle bordered by $\alpha_1^+$, $\alpha_1^-$ and the unstable sides of $R$. 
Note that $$S\subset I(\delta)\subset \Theta.$$

\begin{lemma}\label{curvature2}
Let $\gamma$ be a $C^2(b)$-curve in $I(\delta)$ which admits a critical point.
If $n\geq1$ is such that ${\rm Int}(\Theta)\cap f^i(\gamma)=\emptyset$ for every $i\in\{0,1,\ldots,n-1\}$ and 
${\rm Int}(\Theta)\cap f^n(\gamma)\neq\emptyset$, 
then any connected component of $\Theta\cap f^n(\gamma)$ is a $C^2(b)$-curve.
\end{lemma}
\begin{proof}
By Lemma \ref{curvature} it suffices to show that for any $z\in\gamma$
with $f^n(z)\in\Theta$ and a nonzero vector $v$ tangent to $\gamma$ at $z$,
$\|Df^n(v)\|\geq\delta\|Df^i(v)\|$ holds for every $i\in\{0,1,\ldots,n-1\}$. 
This is a consequence of Lemma \ref{expand} and Proposition \ref{binding}(I)(c).
\end{proof}

\begin{proof}[Proof of Proposition \ref{binding}]
We start with establishing three preliminary estimates.
\medskip

\noindent{\it Estimate 1 (horizontal distance).}
Let $\mathcal F^s(f(\zeta))=\{(x^s(y),y)\}$ denote the leaf of the foliation through $f(\zeta)$ as in Sect.\ref{critical}.
Write $f(z)=(x_0,y_0)$. In the first step we estimate $|x_0-x^s(y_0)|$.

Write $f(\zeta)=(x^s(y_1),y_1)$ and
$e^s(f(\zeta))=
\left(\begin{smallmatrix}
\cos\theta(\zeta)\\
\sin\theta(\zeta)
\end{smallmatrix}\right)$, $\theta(\zeta)\asymp\pi/2$.
Define two functions $\xi=\xi(x,y)$ and $\eta=\eta(y)$ implicitly by
$$\begin{pmatrix}
x\\
 y
\end{pmatrix}=\begin{pmatrix}
x^s(y_1)\\
y_1
\end{pmatrix}+\xi\cdot\begin{pmatrix}
1\\
0
\end{pmatrix}+\eta\cdot\begin{pmatrix}
\cos\theta(\zeta)\\
\sin\theta(\zeta)
\end{pmatrix}.$$
Solving these equations gives
\begin{equation*}
\xi(x,y)=x-x^s(y_1)-\frac{\cos\theta(\zeta)}{\sin\theta(\zeta)}(y-y_1). 
\end{equation*}
A direct computation gives
$$\frac{d^2\xi(x^s(y),y)}{dy^2}=\frac{d\xi}{dx}\frac{d^2x^s(y)}{dy^2}+\frac{d^2\xi}{dx^2}\left(\frac{dx^s(y)}{dy}\right)^2
+2\frac{d^2\xi}{dxdy}\frac{dx^s(y)}{dy}+\frac{d^2\xi}{dy^2}.$$
Using $|\frac{d^2x^s(y)}{dy^2}|\leq C$ and $|\frac{dx^s(y)}{dy}|\leq C\sqrt{b}$
which follow from conditions (c) (d) in Sect.\ref{critical},
$$\left|\frac{d^2\xi(x^s(y),y)}{dy^2}\right|\leq C.$$
Since $f(\gamma)$ is tangent to $\mathcal F^s(f(\zeta))$ at $f(\zeta)$,
$$\frac{d\xi(x^s(y),y)}{dy}(y_1)=0.$$ 
We get
\begin{equation}\label{quad0}|\xi(x^s(y_0),y_0)|=|\xi(x^s(y_0),y_0)-\xi(x^s(y_1),y_1)|\leq C|y_0-y_1|^2.\end{equation}
We also have
\begin{equation}\label{quad-1}|y_0-y_1|\leq |\eta(y_0)|.\end{equation}
Parametrize the $C^2(b)$-curve $\gamma$ by arc length $s$ so that
that $\gamma(s_0)=z$ and $\gamma(s_1)=\zeta$.
Then $\zeta-z=\int_{s_0}^{s_1}D_{\gamma(s)}f(\dot{\gamma}(s))ds$,
where the dot ``$\cdot$" denotes the $s$-derivative.
Split
\begin{equation}\label{coef}D_{\gamma(s)}f(\dot\gamma(s))=A(\gamma(s))\cdot
\begin{pmatrix}1\\0\end{pmatrix}+B(\gamma(s))\cdot \begin{pmatrix}
\cos\theta(\zeta)\\
\sin\theta(\zeta)
\end{pmatrix}.\end{equation}
The proof of \cite[Lemma 2.2]{Tak11} implies $$|A(\gamma(s))|\asymp 2|\gamma(s)-\zeta|\ \text{ and }\
|B(\gamma(s))|\leq C\sqrt{b}.$$
Integrations from $s=s_0$ to $s_1$ gives
\begin{equation}\label{quad1}|\xi(x_0,y_0)|\asymp 2|z-\zeta|^2\ \text{ and }\ |\eta(y_0)|\leq C\sqrt{b}|z-\zeta|.\end{equation}
Using \eqref{quad0} \eqref{quad-1} for $y=y_0$ and the second estimate in \eqref{quad1} we obtain
\begin{equation}\label{quad2}|\xi(x^s(y_0),y_0)|\leq C|y_0-y_1|^2\leq
C|\eta(y_0)|^2 \leq Cb|\xi(x_0,y_0)|.\end{equation} 
This yields
\begin{equation}\label{quad-3}
|x_0-x^s(y_0)|=|\xi(x_0,y_0)-\xi(x^s(y_0),y_0)|\asymp 2|\zeta-z|^2.
\end{equation}
This implies that the tangency between $f(\gamma)$ and $\mathcal F^s(f(\zeta))$ at $f(\zeta)$ is quadratic.
If there were two critical points on $\gamma$, then the two leaves through the critical values intersect
each other.
This is absurd
because the leaves of the foliation $\mathcal F^s$ are integral curves of $C^1$ vector fields.


\medskip

\noindent{\it Estimate 2 (slopes and lengths of iterated curves).}
Let $l$ denote the straight segment connecting $f(z)=(x_0,y_0)$ and $(x^s(y_0),y_0)\in\mathcal F^s(f(\zeta))$.
Arguing inductively, it is possible to show that
for every $i\in\{1,2,\ldots, p-1\}$
the slopes of tangent directions of $f^i(l)$ are everywhere $\leq\sqrt{b}$, and
\begin{equation}\label{lema1}
\begin{split} {\rm length}(f^{i}(l))&\asymp{\rm length}(l)\|w_{i+1}(\zeta)\|  \leq D_{p-1}(\zeta)\|w_{i+1}(\zeta)\|\\
&    \leq D_{p-1}(\zeta)\|w_{p}(\zeta)\| \approx\tau.\end{split}
\end{equation}
The $\asymp$ follows from the bounded distortion near $Q$, and
the first inequality from the definition of the bound period $p$.
The $\approx$ follows from the next estimate:
 using \eqref{eq-2} we have
\begin{equation}\label{lema}
D_k(\zeta)= \tau\left[\sum_{i=1}^k\frac{\|w_{i}(\zeta)\|^2}{\|w_{i+1}(\zeta)\|}\right]^{-1}\asymp
\tau\left[\sum_{i=1}^ke^{\lambda^u(\delta_Q)(i-1)}\right]^{-1}\approx \tau 
e^{-\lambda^u(\delta_Q)k}.\end{equation}
\medskip

\noindent{\it Estimate 3 (length of fold periods).}
Define a {\it fold period} $q=q(\zeta,z)$ by
\begin{equation}\label{Q}
q=\min\{i\in\{1,2,\ldots,p-1\}\colon|\zeta-z|^{\beta}\|w_{j+1}(\zeta)\|\geq1\ \ \text{for every }j\in\{i,i+1,\ldots,p-1\}\},\end{equation}
where
$$\beta=-\frac{1}{\log b}.$$
This definition makes sense because $|\zeta-p|^{\beta}\|w_p(\zeta)\|=|\zeta-z|^{\beta-2}|\zeta-z|^{2}\|w_p(\zeta)\|>1$ from \eqref{lema}.
By the definition of $q$ and \eqref{eq-1} we have
 $$1\leq |\zeta-z|^{\beta}\|w_{q+1}(\zeta)\|\leq|\zeta-z|^{\beta}5^q.$$
 This yields
 \begin{equation}\label{qlength}
q\geq \log|\zeta-z|^{-\frac{\beta}{\log 5}}.
 \end{equation}
\medskip

\noindent{\it Proof of Proposition \ref{binding} (continued).}
Recall that any closed ball of radius $\sqrt{\tau}$ about a point in $\alpha_0^-$ is contained in $U$.
Hence, the conditions $f^{n(\zeta)}(\zeta)\in U^-$, $f^{n(\zeta)+1}(\zeta)\notin U^-$ and the hyperbolicity of the fixed saddle $Q$
altogether imply that there is a ball of radius 
of order $\sqrt{\tau}$ about $f^{n(\zeta)-1}(\zeta)$ which is contained in $U$.
Since $p\leq n(\zeta)$, for every $i\in\{1,2,\ldots,p-1\}$
there is a ball of radius 
of order $\sqrt{\tau}$ about $f^{i}(\zeta)$ which is contained in $U$.
Since
${\rm length}(f^{i-1}(l))<\tau$ as above,
we obtain
$|f^{i}(\zeta)-f^{i}(z)|<{\rm length}(f^{i-1}(l))+(Cb)^{i}\ll\sqrt{\tau},$
and therefore $f^i(z)\in U$ and (a) holds.

Using \eqref{quad-3} \eqref{lema1} and \eqref{eq-2}
we have $$|\zeta-z|^2<D_{p-1}(\zeta)\approx \tau e^{-\lambda^u(\delta_Q)p}.$$
Taking logs of both sides and rearranging the result gives $p\leq -\log|\zeta-z|^{\frac{3}{2\log 2}}$
because $\lambda^u(\delta_Q)\to\log4$ as $b\to0$.
Since $3|\zeta-z|^2>D_{p}(\zeta)$, the lower estimate follows similarly
and (b) holds.

Write $e^s(f(z))=
\left(\begin{smallmatrix}
\cos\theta(z)\\
\sin\theta(z)
\end{smallmatrix}\right)$, $\theta(z)\approx\pi/2$.
Recall that $v$ is any nonzero vector tangent to $\gamma$ at $z$.
Split
\begin{equation}\label{coef2}\frac{1}{\|v\|}\cdot D_zf(v)=A_0\cdot
\begin{pmatrix}1\\0\end{pmatrix}
+B_0\cdot  \begin{pmatrix}
\cos\theta(z)\\
\sin\theta(z)
\end{pmatrix}.\end{equation}
From \eqref{coef} \eqref{coef2} we have
\begin{align*}
A(z)+B\cos\theta(\zeta)&=A_0+B_0\cos\theta(z);\\
B(z)\sin\theta(\zeta)&=B_0\sin\theta(z).
\end{align*}
Solving these equations gives
$$A_0-A(z)=B(z)(\cos\theta(\zeta)-\cos\theta(z))+B(z)\left(1-\frac{\sin\theta(\zeta)}{\sin\theta(z)}\right)\cos\theta(z).$$
The right-hand-side is $\leq |B(z)||\zeta-z|\leq C\sqrt{b}|\zeta-z|$ in modulus, and hence
we have $|A_0|\asymp2|\zeta-z|$.
Therefore
\begin{equation}\label{eq-4}|A_0|\cdot \|D_{f(z)}f^{i-1}\left(\begin{smallmatrix}1\\0\end{smallmatrix}\right)\|\asymp2|\zeta-z|\cdot\|w_i(\zeta)\|
\ \ \text{for every }i\in\{0,1,\ldots,p\}.\end{equation}

Let $i\in\{q,q+1,\ldots,p\}$. From the definition of $q=q(\zeta,z)$ in \eqref{Q},
\begin{equation}\label{eq-5}|\zeta-z|\cdot\|w_i(\zeta)\|\geq|\zeta-z|^{1-\beta}.\end{equation}
On the other hand,
$$\|D_{f(z)}^i(e^s(f(z)))\|\leq(Cb)^i\leq (Cb)^q\leq|\zeta-z|^{\frac{5}{3}}.$$
The last inequality holds for sufficiently small $\delta$,
by virtue of the definition of $q$ and its lower bound in \eqref{qlength}.
Hence
\begin{equation}\label{eq-10}\|D_zf^i(v)\|\asymp 2|\zeta-z|\cdot\|w_i(\zeta)\|\cdot\|v\|\ \ \text{for every }i\in\{q,q+1,\ldots,p\}.
\end{equation}
This yields $$\frac{\|D_zf^p(v)\|}{\|D_zf^i(v)\|}\asymp\frac{\|w_p(\zeta)\|}{\|w_i(\zeta)\|}\asymp e^{\lambda^u(\delta_Q)(p-i)},$$
and therefore
 $\|D_zf^p(v)\|\geq e^{\lambda_0 (p-i)}\|D_zf^i(v)\|$.

Using \eqref{eq-10} with $i=p$ and then \eqref{lema} gives 
\begin{equation}\label{eq-3}\|D_zf^p(v)\|\asymp2|\zeta-z|\cdot\|w_p(\zeta)\|\cdot\|v\|\approx\tau^2|\zeta-z|^{-1}\|v\|
\approx\tau^{\frac{3}{2}}e^{\frac{p}{2}\lambda^u(\delta_Q)}\|v\|,\end{equation}
and therefore
 $\|D_zf^p(v)\|\geq e^{\lambda_0 p}\|v\|$ provided $\delta$ is sufficiently small and hence $p$ is large.

We now treat the case $i\in\{1,2,\ldots,q-1\}$.
Using $\|w_i(\zeta)\|<\|w_q(\zeta)\|$
and the definition of $q$ we have
$$|\zeta-z|\cdot\|w_i(\zeta)\|\leq |\zeta-z|\cdot\|w_q(\zeta)\|<|\zeta-z|^{1-\beta}<\sqrt{\delta}.$$
For the other component in the splitting, 
$$\|D_{f(z)}^i(e^s(f(z)))\|\leq(Cb)^i\leq Cb.$$
Hence $\|D_zf^i(v)\|<\|v\|$. Using this and \eqref{eq-3} we obtain
$$\|D_zf^p(v)\|>\frac{\|D_zf^i(v)\|}{\|v\|}\|D_zf^p(v)\|\asymp\|D_zf^i(v)\||\zeta-z|^{-1}\approx\frac{1}{\sqrt{\tau}} \|D_zf^i(v)\|e^{\frac{\lambda^u(\delta_Q)}{2}p}.$$
This yields $\|D_zf^p(v)\|\geq e^{\lambda_0 (p-i)}\|D_zf^i(v)\|$ provided $\delta$ is sufficiently small.
We have proved (c).

It is left to prove
(II). By the definition of $n(\zeta)$ we have $f^{n(\zeta)}(\zeta)\in U^-$ and $f^{n(\zeta)+1}(\zeta)\notin U^-$.
This and the choice of $\tau$ together imply that there is a closed ball of radius of order $\sqrt{\tau}$
about $f^{n(\zeta)}(\zeta)$ which does not intersect $\alpha_0^-$. 
Since $\zeta\in S$, $f^{n(\zeta)}(\zeta)$ is at the left of $\alpha_0^-$.
Since
${\rm length}(f^{n(\zeta)-1}(l))<\tau$ 
we have
$|f^{n(\zeta)}(\zeta)-f^{n(\zeta)}(z)|<   {\rm length}(f^{n(\zeta)-1}(l))+(Cb)^{n}\ll\sqrt{\tau}.$
This implies $f^{n(\zeta)}(z)\notin R$.
\end{proof}

\section{Global construction}
In this section we use the results in Sect.2 to construct an induced system with uniformly hyperbolic behavior.
From the induced system we extract a hyperbolic set, the dynamics on which is conjugate to the
 full shift on a finite number of symbols. This hyperbolic set will be used to complete the proof of the main theorem in the next section.

\subsection{Construction of induced system}\label{construct}
In this subsection we deliberately construct an induced system with uniformly hyperbolic Markov structure,
with countably infinite number of branches.
Although a similar construction has been done in \cite{Tak15} at $a=a^*$ to analyze
equilibrium measures for the potential $-t\log J^u$ as $t\to+\infty$, it
does not fit to our purpose of studying the case $t\to-\infty$. 
Moreover, a treatment of the case $a^{**}<a<a^*$ brings additional difficulties which 
are not present in \cite{Tak15}.
We exploit the geometric structure of invariant manifolds of $P$ and $Q$ which 
are ``not destroyed yet" by the homoclinic bifurcation.

We say a $C^2(b)$-curve $\gamma$ in $R$ {\it stretches across} $\Theta$ if both endpoints of $\gamma$ 
are contained in the stable sides of $\Theta$.
Let $\omega$, $\omega'$ be two rectangles in $\Theta$ such that $\omega\subset\omega'$.
We call $\omega$ a {\it $u$-subrectangle} of $\omega'$ 
if each stable side of $\omega'$ contains one stable side of $\omega$.
Similarly, we call $\omega$ an {\it $s$-subrectangle} of $\omega'$ 
if each unstable side of $\omega'$ contains one unstable side of $\omega$.

\begin{prop}\label{induce}
There exist a $u$-subrectangle $\Theta'$ of $\Theta$, a large integer $k_0\geq1$, a constant $C\in(0,1)$
and a countably infinite family $\{\omega_k\}_{k\geq k_0}$ of $s$-subrectangles of $\Theta'$ with the following properties:

\begin{itemize}

\item[(a)] the unstable sides of $\Theta'$ are $C^2(b)$-curves stretching across $\Theta$
and intersecting ${\rm Int}(S)$;

\item[(b)] for every nonzero vector $v$ tangent to the unstable side of $\Theta'$ and every integer $n>0$,
$\|Df^{-n}(v)\|\leq e^{-\lambda_0 n}\|v\|;$


\item[(c)] for each $\omega_k$, ${\rm Int}(\Theta)\cap f^i(\omega_k)=\emptyset$ for every $i\in\{1,2,\ldots,k-1\}$,
and
$f^k(\omega_k)$ is a $u$-subrectangle of $\Theta'$ whose unstable 
sides are $C^2(b)$-curves stretching across $\Theta'$;

\item[(d)] if $z\in\omega_k\cap\Omega$ 
and ${\rm slope}(E_z^u)\leq\sqrt{b}$, then
$\|Df^{k}|E_{z}^u\|\geq Ce^{\lambda^u(\delta_Q)k}.$

\end{itemize}
\end{prop}

 \begin{proof}
Since the construction is involved, we start with giving a brief sketch. 
 Let $\gamma_0$ denote the $C^2(b)$-curve in $W^u$ which stretches across $\Theta$ and is part of the boundary of $S$. 
 This curve, which obviously satisfies the exponential backward contraction property as in item (b), will be one of the unstable sides of $\Theta'$. 
 In Step 1 we find another $C^2(b)$-curve stretching across $\Theta$,
 which will be the other unstable side of $\Theta'$.
In Step 2 we construct the rectangles $\{\omega_k\}_{k\geq k_0}$ by subdividing $\Theta'$
into smaller rectangles, with a family of compact curves in $W^s(P)$.
Both of the steps depend on the orientation of the map $f$.
\medskip

\noindent{\it Step 1 (construction of $\Theta'$).}
Suppose that $\alpha,\alpha'$ are two compact curves in $R\cap W^s(P)$ which join the two unstable 
sides of $R$ and intersect $\gamma_0$ exactly at one point.
We write $\alpha\prec\alpha'$ if 
${\rm proj}(\alpha\cap \gamma_0)<{\rm proj}(\alpha'\cap \gamma_0)$,
where ${\rm proj}$ denotes the projection to the first coordinate.

Set  $\tilde\alpha_0=\alpha_1^+$ and $\tilde\alpha_1=\alpha_1^-$. 
Let $\{\tilde\alpha_k\}_{k\geq0}$ denote the sequence of compact curves in $R\cap W^s(P)$
with the following properties: 
each $\tilde\alpha_k$ joins the two unstable sides of $R$; $\tilde\alpha_k\prec\tilde\alpha_{k-1}$ and
$ f(\tilde\alpha_k)\subset\tilde\alpha_{k-1}$
for every $k\in\{1,2,\ldots\}.$
Notice that $\tilde\alpha_k$ converges to $\alpha_0^-$ as $k\to\infty$.

For every $k\geq1$
the set $R\cap f^{-2}(\tilde\alpha_{k-1})$ has three or four connected components.
Two of them are $\tilde\alpha_{k+1}$ and the connected component of
$R\cap f^{-1}(\tilde\alpha_{k})$ which is not $\tilde\alpha_{k+1}$.
Let $\alpha_{k}$ denote the union of the remaining one or two connected components
of $R\cap f^{-2}(\tilde\alpha_{k-1})$.
By definition, $\alpha_{k}$ is located near the origin
(If $a=a^*$, then for every $k\geq0$, 
$\alpha_{k}$ has two connected components.
If $a^{**}<a<a^*$ then there exists an integer $k'=k'(a)$ such that 
$\alpha_{k}$ has two connected components if and only if $k< k'$).
Choose a large integer $\hat k\geq1$ such that $\alpha_k\subset I(\delta)$ holds for every $k\geq \hat k$.


The rest of the construction of $\Theta'$ depends on the orientation of $f$.
We first consider the case $\det Df>0$.
Let $k\geq \hat k$. The set $\alpha_k$ intersects $\gamma_0$ exactly at two points.
Let $\gamma_k^-,\gamma_k^+$ denote the compact curve in $\gamma_0$ whose endpoints are 
in $\alpha_k$ and $\alpha_{k+1}$, and satisfy $\sup {\rm proj}(\gamma_k^-)<\inf {\rm proj}(\gamma_k^+)$.
Since ${\rm Int}(\Theta')\cap f^i(\gamma^k)=\emptyset$ for every $i\in\{1,2,\ldots,k-1\}$
and $f^k(\gamma^k)\subset\Theta$,
 $f^k(\gamma^k)$ is a $C^2(b)$-curve from Lemma \ref{curvature2}.
 In addition, since its endpoints are contained in the stable sides of $\Theta'$,
 $f^k(\gamma^k)$ stretches across $\Theta$.
Enlarging $\hat k$ if necessary, we have ${\rm Int}(S)\cap f^k(\gamma_k^+)\neq\emptyset$
for every $k\geq \hat k$.
Define $\Theta'$ to be the rectangle bordered by $\gamma_0$,
$f^{\hat k}(\gamma_{\hat k}^+)$ and the stable sides of $\Theta$.
The exponential backward contraction property in item (b) may be proved along the line 
of the proof of \eqref{backward} and hence we omit it.

\begin{remark}
{\rm There is no particular reason for our choice of $\gamma_k^+$. Choosing $\gamma_k^-$ does the same job.}
\end{remark}

The case $\det Df<0$ is easier to handle.
Since $a\in(a^{**},a^*]\subset\mathcal G$, there is the unique compact domain bounded by $f^{-2}(W^s_{\rm loc}(Q))$ and $\ell^u$,
which is contained in $S$.
Define $\Theta'$ denote the rectangle bordered by the stable sides of $\Theta$,
$\gamma_0$ and the $C^2(b)$-curve in $W^u(Q)$ which stretches across $\Theta$ and contains $\ell^u$.
Items (a) and (b) in Proposition \ref{induce} hold.

\medskip

\noindent{\it Step 2 (construction of $\omega_k$).}
The set $\alpha_{k}\cap \Theta'$ consists of two connected components, one at the left
of $S$ and the other at the right of $S$. 
Let $\hat\omega_k$ denote the connected component of $\alpha_{k}\cap \Theta'$ 
which lies at the right of $S$. 
Then $f^k(\hat\omega_k)$ is a $u$-subrectangle of $\Theta$, whose unstable sides 
are $C^2(b)$-curves stretching across $\Theta$.

In the case $\det Df>0$, 
choose a sufficiently large integer $k_0>\hat k$
depending on the parameter $a$ such that for every $k\geq k_0$,
 $f^k(\hat\omega_k)$ is a $u$-subrectangle of $\Theta'$.
 Set $\omega_k=\hat\omega_k$.
 In the case $\det Df<0$, $f^k(\hat\omega_k)$ may not be contained in $\Theta'$ (See FIGURE 6),
and this is always the case for $a<a^*$ and sufficiently large $k$.
However, note that 
$f^{k+2}(\hat\omega_k)$ contains a unique $u$-subrectangle $\omega'$ of $\Theta'$.
Set  $k_0=\hat k+2$ and
$\omega_{k}=f^{-k}(\hat\omega_{k-2})$ for every $k\geq k_0$.
This finishes the construction of $\{\omega_k\}_{k\geq k_0}$. Item (c) is a direct consequence of the construction.

To prove item (d) we need
 the next uniform upper bound
on the length of bound periods.
\begin{lemma}\label{upbound}
There is a constant $E>0$ such that if $z\in \bigcup_{k\geq k_0}\omega_k$ and $\zeta$ is a critical point
on a $C^2(b)$-curve which is tangent to $v$, 
then $p(\zeta,z)\leq E$.
\end{lemma}
\begin{proof}
Let $z$, $\zeta$ be as in the statement of the lemma and assume $z\in\omega_k$.
By construction, one of the unstable sides of $\omega_k$ is contained in the $C^2(b)$-curve $\gamma_0$ which 
is not contained in the unstable sides of $\Theta$ and stretches across $\Theta$.
Let $\zeta'$ denote the critical point on $\gamma_0$.
With a slight abuse of notation, let $\mathcal F^s(\alpha_0^+)$ denote the leaf containing $\alpha_0^+$.
The leaf $\mathcal F^s(f(\zeta'))$ lies at the right of $\mathcal F^s(\alpha_0^+)$, and 
$\mathcal F^s(f(\zeta))$ lies at the right of $\mathcal F^s(\alpha_0^+)$.
Since $f(z)\in R$ we have
$$3|\zeta-z|^2\geq \inf\{|z_1-z_2|\colon z_1\in\mathcal F^s(\alpha_0^+),\ z_2\in\mathcal F^s(f(\zeta'))\}>0.$$
Taking logs of both sides and then using Proposition \ref{binding}(a) 
 yields the claim.
 \end{proof}

Since the point $z$ is sandwiched by the two
 $C^2(b)$-curves intersecting ${\rm Int}(\Theta)$,
there exists a $C^2(b)$-curve which is tangent to $E_z^u$
and contains a critical point $\zeta$.
Let $p=p(\zeta,z)$ denote the bound period given by Proposition \ref{binding}.
Since ${\rm slope}(E_{f^{p}(z)}^u)\leq\sqrt{b}$,
$f^n(z)\notin{\rm Int}(\Theta)$ for every $n\in\{n,n+1,\ldots,k-1\}$ and $f^k(z)\in\Theta$,
the bounded distortion for iterates near $Q$ gives
$$\|D_{f^{p}(z)}f^{k-p}|E_{f^{p}(z)}^u\|\approx  e^{\lambda^u(\delta_Q)(k-p)}.$$
Using $\|Df^{p}|E_z^u\|>1$ and $p\leq E$ by Lemma \ref{upbound} we obtain
$$\|Df^{k}|E_{z}^u\|=\|Df^{p}|E_{z}^u\|\cdot\|Df^{k-p}|E_{f^{p}(z)}^u\|\geq 
Ce^{-\lambda^u(\delta_Q)E}\|D_{f^{p}(z)}f^{k-p}|E_{f^{p}(z)}^u\|,$$
and hence item (d) holds.
This completes the proof of Proposition \ref{induce}.
\end{proof}

\begin{figure}
\begin{center}
\includegraphics[height=5.5cm,width=14.5cm]
{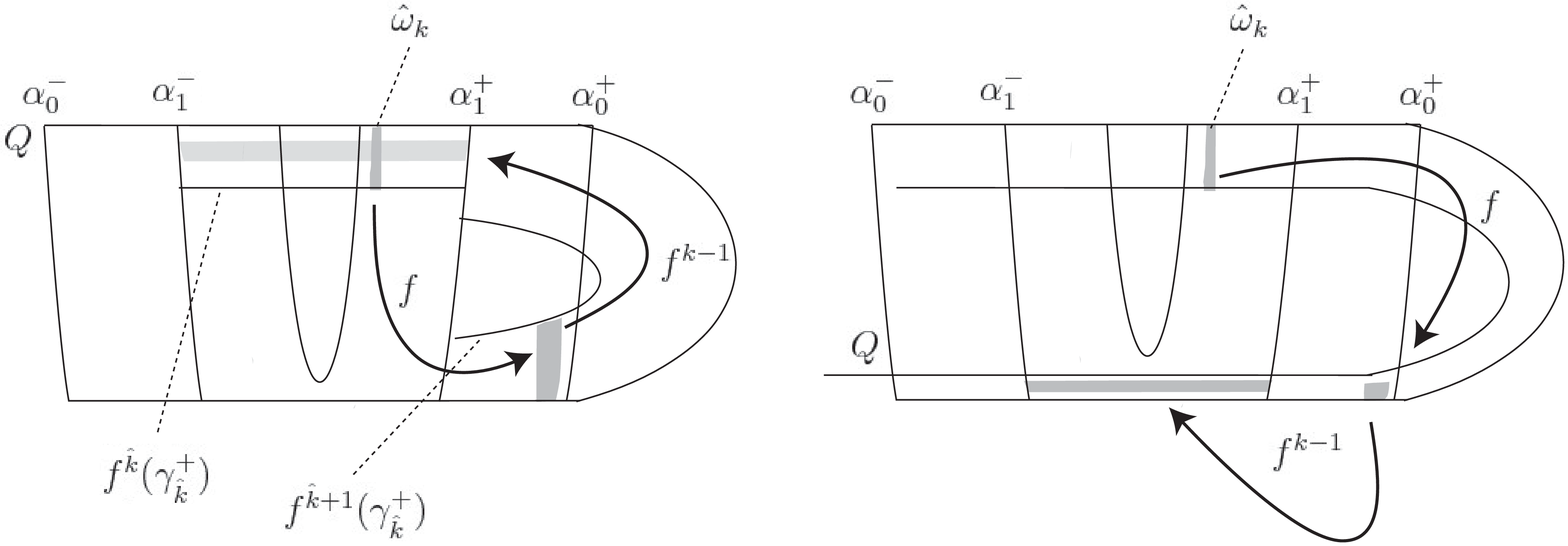}
\caption{The rectangles $\hat\omega_k$, $f(\hat\omega_k)$, $f^k(\hat\omega_k)$ (shaded) for $a\in(a^{**},a^*)$: 
the case $\det Df>0$ (left); the case $\det Df<0$ (right).}
\end{center}
\end{figure}





\subsection{Symbolic dynamics}\label{symbol}
From the induced system in Proposition \ref{induce}
we extract a finite number of branches, and construct a conjugacy to the full shift on a finite number of symbols. 
For two positive integers $q_0$, $q_1$ with $q_0<q_1$ define
 $$\Sigma(q_0,q_1)=\{\underline{a}=\{a_i\}_{i\in\mathbb Z}\colon a_i\in\{q_0,q_0+1,\ldots,q_1\}\}.$$
 This is the set of two-sided sequences with $q_1-q_0+1$-symbols.
 Endow $\Sigma(q_0,q_1)$ with the product topology of the discrete topology of $\{q_0,q_0+1,\ldots,q_1\}$.

\begin{prop}\label{symbolprop}
For all integers $1\leq q_0<q_1$ there exist a continuous injection $\pi\colon\Sigma(q_0,q_1)\to\Omega$ 
and a constant $C\in(0,1)$ such that the following holds:

\begin{itemize}

\item[(a)]  for every $\underline{a}\in\Sigma(q_0,q_1)$,
${\rm slope}(E_{\pi(\underline{a})}^u)\leq\sqrt{b};$

\item[(b)]  for every $\underline{a}=\{a_i\}_{i\in\mathbb Z}\in\Sigma(q_0,q_1)$,
 $\|Df^{a_0}|E^u_{\pi(\underline{a})}\|\geq  Ce^{\lambda^u(\delta_Q)a_0}$
and $\|Df^{a_i}|E^u_{f^{a_0+a_1+\cdots+a_{i-1}}(\pi(\underline{a}))}\|\geq  Ce^{\lambda^u(\delta_Q)a_i}$
for every integer $i\geq1$;

\item[(c)] $\underline{a}\in\Sigma(q_0,q_1)\mapsto E^u_{\pi(\underline{a})}$ is continuous.
\end{itemize}
\end{prop}

\begin{proof}
 Let $\underline{a}=\{a_i\}_{i\in\mathbb Z}\in\Sigma(q_0,q_1)$.
 For each integer $j\geq1$ define
 $$\omega^s_j=\omega_{a_0}\cap\left(\bigcap_{i=1}^j 
f^{-a_0}\circ f^{-a_1}\circ\cdots\circ f^{-a_{i-1}}(\omega_{a_i})\right)$$
and 
$$\omega^u_j=\bigcap_{i=1}^{j}f^{a_{-1}}\circ f^{a_{-2}}\circ\cdots\circ f^{a_{-i}}(\omega_{a_{-i}}).$$
Note that $\{\omega_j^s\}_{j\geq1}$ is a decreasing sequence of $s$-subrectangles of $\Theta'$,
and $\{\omega_j^u\}_{j\geq1}$ is a decreasing sequence of $u$-subrectangles of $\Theta'$.
Define a coding map $\pi\colon\Sigma(q_0,q_1)\to\Omega$ by
$$\{\pi(\underline{a})\}=\left(\bigcap_{j=1}^{+\infty}\omega_j^s\right)
\cap\left(\bigcap_{j=1}^{+\infty}\omega_j^u\right).$$ 
 We show below that the right-hand-side is a singleton, and so
$\pi$ is well-defined.

By Corollary \ref{c2cor} and the fact that $f$ contracts area, the set $\bigcap_{j=1}^{+\infty}\omega_j^u$ 
is a $C^1$ curve which connects the stable sides of $\omega_{a_0}$.
By Corollary \ref{c2cor} again, for any nonzero vector $v$ tangent to
 $\bigcap_{j=1}^{+\infty}\omega_j^u$ there exists a $C^2(b)$-curve which is tangent to $v$
 and contains a critical point. By
Proposition \ref{binding}(b) and Lemma \ref{expand},
for every integer $m\geq1$ we have
$${\rm length}\left(\left(\bigcap_{j=1}^{m}\omega_j^s\right)\cap\left(\bigcap_{j=1}^{+\infty}\omega_j^u\right)\right)\leq
\exp\left(-\lambda_0\sum_{i=0}^m a_i \right).$$
The right-hand  side goes to $0$ as $m\to\infty$.
This means that $\pi$ is well-defined.

The continuity of $\pi$ is obvious.
To show the injectivity, assume $\underline{a}\neq\underline{a'}$
and $\pi(\underline{a})=\pi(\underline{a'})$.
Then there exists an integer $i$ such that
$f^i(\pi(\underline{a}))$ belongs to two rectangles in $\{\omega_k\}_{k=q_0}^{q_1}$,
namely belongs to a curve in $\cup_{k\geq k_0}\alpha_k$
which is a stable side of two neighboring rectangles.
It follows that $f^{a_0+a_1+\cdots+a_i}(\pi(\underline{a}))\in \alpha_1^+$ holds for all large integer $i>0$.
On the other hand, the definition of $\pi$ gives
$f^{a_0+a_1+\cdots+a_i}(\pi(\underline{a}))\in\omega_{a_i}$.
Since $\omega_{a_i}\cap\alpha_1^+=\emptyset$ by construction, we obtain a contradiction.

Recall that the unstable direction is characterized by the exponential backward contraction property \eqref{eu}.
By Corollary \ref{c2cor} and the fact that the $C^1$-curve $\bigcap_{j=1}^{+\infty}\omega_j^u$ is obtained as the $C^1$-limit of the unstable sides of 
$\{\omega_j^u\}_{j\geq1}$.
To prove items (a) and (c) it suffices to show that 
for every integer $j\geq1$, any unstable side $\gamma$ of $\omega_j^u$,
every $z\in\gamma$, every integer $n\geq0$ and every vector $v$
 tangent to $f^{-n}(\gamma)$ at $f^{-n}(z)$,
\begin{equation}\label{backward}\|D_{f^{-n}(z)}f^n(v)\|\geq e^{\lambda_0 n}\|v\|.\end{equation}
Then Item (b) is a consequence of the construction of $\pi$ and Proposition \ref{induce}(d).

It is left to prove \eqref{backward}. For each $i\in\{1,2,\ldots,j\}$ define an integer $n_i\leq0$ by $f^{n_i}(\gamma)\in\omega_{a_i}$.
Note that $n_1<n_2<\ldots<n_{j-1}<n_j=0$.
Below we treat four cases separately.


\smallskip

\noindent{\it Case I: $-n=n_i$ for some $i$.}
We split the time interval $[n_i,0]$ into subintervals $[n_l,n_{l+1}]$ $(l=i,i+1,\ldots,j-1)$.
Then we apply the derivative estimates in Lemma \ref{expand} and Proposition \ref{binding}(c) 
to $[n_l,n_l+p_l]$ and $[n_l+p_l,n_{l+1}]$ respectively.
Recall that $\lambda^u(\delta_Q)\to\log4$ as $b\to0$
and $\lambda_0=\frac{99}{100}\log2$.
We obtain
\begin{equation}\label{back1}
\|D_{f^{-n}(z)}f^{n}(v)\|\geq\left( \prod_{l=i}^{j-1}e^{\lambda_0(n_{l+1}-n_l)}\right)\|v\|=e^{\lambda_0(n_j-n_i)}\|v\|=
e^{\lambda_0n}\|v\|.\end{equation}
\smallskip

\noindent{\it Case II: $n_i+p_i\leq -n<n_{i+1}$ for some $i$.}
Since $f^{n_{i+1}}(z)\in I(\delta)$, Lemma \ref{expand} gives
$$\|D_{f^{-n}(z)}f^{n_{i+1}+n}(v)\|\geq e^{\lambda_0(n_{i+1}+n)}\|v\|.$$
Using this and \eqref{back1} with $-n=n_{i+1}$ we get
$$\|D_{f^{-n}(z)}f^n(v)\|=\frac{\|D_{f^{n_{i+1}}(z)}f^{-n_{i+1}}(D_{f^{-n}(z)}f^{n_{i+1}+n}(v))\|}{\|D_{f^{-n}(z)}f^{n_{i+1}+n}(v)\|}
\cdot\|D_{f^{-n}(z)}f^{n_{i+1}+n}(v)\|\geq e^{\lambda_0 n}\|v\|.$$

\noindent{\it Case III: $n_i<-n<n_i+p_i$.}
Proposition \ref{binding}(c) gives
$$\|D_{f^{-n}(z)}f^{n_i+p_i+n}(v)\|\geq e^{\lambda_0(n_i+p_i+n)}\|v\|.$$
Combining this with the result in Case II yields the desired inequality.
\medskip

\noindent{\it Case IV: $-n<n_1$.}
 Since $f^{n_1}(\gamma)$ is contained in the unstable sides of $\Theta'$, Proposition \ref{induce}(b) gives
$\|D_{f^{-n}(z)}f^{n_1+n}(v)\|\geq e^{\lambda_0(n_1+n)}\|v\|$.
From this and \eqref{back1} with $-n=n_1$ we get the desired inequality.
\end{proof}


\section{Proof of the Main Theorem}
In this section we use the results in Sect.2 and finish the proof of the main theorem. 
Finally we provide more details on the main theorem,
on the abundance of parameters beyond $a^*$ satisfying $\mathcal M_0(f)=\mathcal M(f)$.ƒ

\subsection{Proof of the main theorem}
By virtue of Lemma \ref{maximal}, the removability in the main theorem follows from the next

\begin{prop}\label{estcor}
For any $t<0$ there exists a measure $\mu\in\mathcal M(f)$ such that
\begin{equation*}
h_\mu(f)-t\lambda^u(\mu)>-t\lambda^u(\delta_Q).
\end{equation*}

\end{prop}

\begin{proof}
Let $q>0$ be the square of a large integer, to be determined later depending on $t$.
Set $\Sigma(q)=\Sigma(q-\sqrt{q}+1,q)$, and
let $\sigma\colon\Sigma(q)\circlearrowleft$ denote the left shift.
For each $\underline{a}\in\Sigma(q)$ define
$r(\underline{a})=\sum_{i=0}^{q-1}a_i$. 
Given a $\sigma^q$-invariant Borel probability measure $\mu$,
define a Borel measure $\mathcal L(\mu)$ by
$$\mathcal L(\mu)=\frac{1}{\int rd\mu}\sum_{[a_0,a_1,\ldots,a_{q-1}]}\sum_{i=0}^{a_0+a_1+\cdots+a_{q-1}-1} f_*^i(\pi_*(\mu |_{[a_0,a_1,\ldots,a_{q-1}]})),$$
where $[a_0,a_1,\ldots,a_{q-1}]=\{\underline{b}\in\Sigma(q)\colon b_i=a_i,\  i\in\{0,1,\ldots,q-1\}\}$
and $\pi\colon\Sigma(q)\to\Omega$ is the coding map given by Proposition \ref{induce}.
Then $\mathcal L(\mu)$ is an $f$-invariant and a probability. 
For $t<0$ define 
 $\Phi_t\colon\Sigma(q)\to\mathbb R$ by 
$$\Phi_t(\underline{a})=-t\log\|Df^{r(\underline{a})}|E^u_{\pi(\underline{a})}\|.$$
From Proposition \ref{symbolprop}, $\Phi_t$ is continuous and satisfies
\begin{align*}\Phi_t(\underline{a})&\geq -tq\log C -tr(\underline{a})\lambda^u(\delta_Q).\end{align*}
Since $q^2/2\leq r(\underline{a})\leq q^2$ and $0<C<1$ we have
\begin{equation*}
\frac{\Phi_t(\underline{a})}{r(\underline{a})}\geq \frac{-2t\log C}{q} -t\lambda^u(\delta_Q).\end{equation*}
Let $\mu_0$ denote the measure of maximal entropy of $\sigma^q$.
For each integer $n\geq1$
set $P_n=\{\underline{a}\in\Sigma(q)\colon \sigma^{qn}(\underline{a})=\underline{a}\}.$ 
Since $r$ and $\Phi_t$ are continuous,
as $n\to\infty$ we have
\begin{equation*}
\frac{1}{\#P_n}\sum_{\underline{a}\in P_n} r(\underline{a})\to \int r d\mu_0\ \text{ and }\
 \frac{1}{\#P_n}\sum_{\underline{a}\in P_n} \Phi_t(\underline{a})\to\int \Phi_t d\mu_0.\end{equation*}
 It follows that
$$\frac{\int\Phi_td\mu_0}{\int r d\mu_0} \geq \frac{-2t\log C}{q} -t\lambda^u(\delta_Q),$$
for otherwise 
we would obtain a contradiction.
Since the entropy of $(\sigma^q,\mu_0)$ is $q\log\sqrt{q}$ and $\int rd\mu_0\leq q^2$
we obtain
\begin{align*}
P(-t\log J^u)\geq h(\mathcal L(\mu_0))-t\int \log J^u d\mathcal L(\mu_0)&=\frac{1}{\int r d\mu_0}
\left(q\log\sqrt{q}+\int \Phi_td\mu_0\right)\\
&\geq \frac{1}{q}\log\sqrt{q}-\frac{2t\log C}{q}-t\lambda^u(\delta_Q)>-t\lambda^u(\delta_Q).
\end{align*}
The strict inequality in the last line holds provided $q>e^{4t\log C}$.
\end{proof}

\begin{remark}
{\rm It is not hard to show that the set $\pi(\Sigma(q))$ is a hyperbolic set. However we do not need this fact.}
\end{remark}

To finish, it is left to show $P(-t\log J^u)=-t\lambda_M^u+o(1)$ as $t\to-\infty$ provided $a=a^*$.
According to \cite{Tak15} let us call 
a measure $\mu\in\mathcal M(f)$  a {\it $(-)$-ground state}
if there exists a sequence $\{t_n\}_n$, $t_n\searrow-\infty$ such that $\mu_{t_n}$ is an equilibrium measure for $-t_n\log J^u$
and $\mu_{t_n}$ converges weakly to $\mu$ as $n\to\infty$.
If $P(-t\log J^u)+\lambda_M^ut\nrightarrow0$ as $t\to-\infty$, 
then the upper semi-continuity of entropy (see \cite{SenTak1}) would imply the existence of a $(-)$-ground state with positive entropy.
If $\det Df>0$, we obtain a contradiction to \cite[Thereom A(b)]{Tak15} which states that the Dirac measure at $Q$ is the unique $(-)$-ground state.
Even if $\det Df<0$, the proof of \cite[Thereom A(b)]{Tak15} works and we obtain the same contradiction. 
This completes the proof of the main theorem. \qed

\subsection{Abundance of parameters satisfying $\mathcal M_0(f)=\mathcal M(f)$}\label{equal}
In the main theorem we have reduced the class of measures to consider:
only those measures which give full weight to the Borel set $\Lambda$ were taken into consideration.
To claim that $\mathcal M_0(f)=\mathcal M(f)$ holds for many parameters except $a^*$,
some preliminary discussions are necessary.

From the Oseledec theorem \cite{Ose68} and the two-dimensionality of the system,
one of the following holds for each measure $\mu\in\mathcal M(f)$ which is ergodic:

\begin{itemize}
\item[(a)] there exist a real number $\chi(\mu)$ such that
for $\mu$-a.e. $z\in\Omega$ and for any vector $v\in T_z\mathbb R^2\setminus\{0\}$,
$$\lim_{n\to\pm\infty}\frac{1}{n}\log\|Df^n(v)\|=\chi(\mu)\ \text{and}\ 
\int\log|\det Df|d\mu=2\chi(\mu);$$

\item[(b)]
there exist two real numbers $\chi^1(\mu)<\chi^2(\mu)$ and for $\mu$-a.e. $z\in\Omega$
a non-trivial splitting $T_z\mathbb R^2=E^1_z\oplus E^2_z$ such that
 for any vector $v^i\in E^i_z\setminus\{0\}$ $(i=1,2)$,
$$\lim_{n\to\pm\infty}\frac{1}{n}\log\|Df^n(v^i)\|=\chi^i(\mu)\ (i=1,2)\ \text{and}\ 
\int\log|\det Df|d\mu=\chi^1(\mu)+\chi^2(\mu).$$
\end{itemize}
We say $\mu$ is a {\it hyperbolic measure}
if (b) holds and $\chi^1(\mu)<0<\chi^2(\mu)$.

\begin{lemma}\label{m0}
Every $f$-invariant ergodic Borel probability measure is a
hyperbolic measure if and only if $\mathcal M_0(f)=\mathcal M(f)$.
\end{lemma}
\begin{proof}
Let $\mu\in \mathcal M(f)$ be ergodic. Then $\mu(\Lambda)=1$ if and only if $\mu$ is a hyperbolic measure,
see e.g., \cite{Kat80}.
The ``if" part follows from this.
The ``only if" part is a consequence of the ergodic decomposition of invariant Borel probability measures.
\end{proof}

 It was proved in \cite{Tak12,Tak13} that if additionally $\{f_a\}$ is $C^4$ in $a,x,y$, then for sufficiently small $b>0$ there exists a set $\Delta$ of $a$-values 
in $(a^{**},a^*]$
 containing $a^*$ 
with the following properties: 
\begin{itemize}
\item $\displaystyle{\lim_{\varepsilon\to+0}} (1/\varepsilon){\rm Leb}(
\Delta\cap[a^*-\varepsilon,a^*])=1$, where {\rm Leb}$($$\cdot$$)$ denotes the one-dimensional Lebesgue measure;

\item if $a\in\Delta$, then the Lebesgue measure of the set $\{z\in\mathbb R^2\colon \text{$\{f^n(z)\}_{n\in\mathbb N}$ is bounded}\}$
is zero;

\item  if $a\in\Delta$, then any ergodic measure is a hyperbolic measure.

\end{itemize}
In other words, the dynamics for parameters in $\Delta$ is like Smale's horseshoe.
However, whether or not the dynamics is uniformly hyperbolic for $a\in\Delta\setminus\{a^*\}$  is a wide open problem.
We even do not know if there exists an increasing sequence of uniformly hyperbolic parameters
in $\Delta$ converging to $a^*$.

\subsection*{Acknowledgments}
Partially supported by the Grant-in-Aid for Young Scientists (A) of the JSPS, Grant No.15H05435
and the JSPS Core-to-Core Program ``Foundation of a Global Research
Cooperative Center in Mathematics focused on Number Theory and Geometry".

\bibliographystyle{amsplain}

\begin{thebibliography}{10}


\bibitem{BedSmi06} Bedford, E. and Smillie, J.: Real polynomial
diffeomorphisms with maximal entropy: II. small Jacobian.
Ergodic Theory and Dynamical Systems {\bf 26}, 1259--1283 (2006)

\bibitem{BenCar91} Benedicks, M. and Carleson, L.: The dynamics of the
H\'enon map. Ann. Math. {\bf 133}, 73--169 (1991)

\bibitem{Bow75} Bowen, R.: {\it Equilibrium states and the ergodic
theory for Anosov diffeomorphisms,} Springer Lecture Notes in
Math. {\bf 470} (1975)


\bibitem{Cao99} Cao, Y.: The non wandering set of some H\'enon map. Chinese Sci. Bull. {\bf 44}, 590--594 (1999) 

\bibitem{CaoMao00} Cao, Y. and Mao, J. M.: The non-wandering set of some H\'enon maps.
Chaos Solitons Fractals {\bf 11}, 2045--2053 (2000)


\bibitem{CLR08} Cao, Y., Luzzatto, S. and Rios, I.: The boundary of
hyperbolicity for H\'enon-like families. Ergodic Theory and
Dynamical Systems {\bf 28}, 1049--1080 (2008)



\bibitem{DevNit79} Devaney, R. and Nitecki, Z.: Shift automorphisms in the
H\'enon mapping. Commun. Math. Phys. {\bf 67}, 137--146 (1979)

\bibitem{Dob09} Dobbs, N.: Renormalisation-induced phase transitions for unimodal maps.
Commun. Math. Phys. {\bf 286}, 377--387 (2009)

\bibitem{Kat80} Katok, A.:
Lyapunov exponents, entropy and periodic orbits for diffeomorphisms.
Publ. Math. Inst. Hautes \'Etud. Sci. {\bf 51}, 137--173 (1980)



\bibitem{Lep11} Leplaideur, R.: Thermodynamic formalism for a family of nonuniformly hyperbolic horseshoes and the unstable Jacobian. Ergodic Theory and Dynamical Systems {\bf 31}, 423--447  
(2011)




\bibitem{Lop90} Lopes, A.: Dynamics of real polynomials on the plane and triple point phase transition.
Math. Comput. Modelling. {\bf 13}, 17--31 (1990)




\bibitem{MakSmi00} Makarov, N. and Smirnov, S.: On ``thermodynamics" of rational maps I. Negative Spectrum.
Commun. Math. Phys. {\bf 211}, 705--743 (2000)




\bibitem{MorVia93} Mora, L. and Viana, M.: Abundance of strange attractors.
Acta Math. {\bf 171}, 1--71 (1993)





\bibitem{Ose68} Oseledec, V.: A multiplicative ergodic theorem:
Lyapunov characteristic numbers for dynamical systems,
Trans. Moskow Math. Soc. {\bf 19}, 197--231 (1968)

\bibitem{PalTak93}
Palis, J. and Takens, F.: {\it Hyperbolicity \& sensitive chaotic
dynamics at homoclinic bifurcations.} Cambridge Studies in
Advanced Mathematics {\bf 35}. Cambridge University Press, 1993







\bibitem{Rue78} Ruelle, D.: {\it Thermodynamic formalism. The mathematical structures of classical equilibrium statistical mechanics.} Encyclopedia of Mathematics and its Applications, 5. Addison-Wesley Publishing Co. 



\bibitem{SenTak1} Senti, S. and Takahasi, H.:
Equilibrium measures for the H\'enon map at the first bifurcation. Nonlinearity {\bf 26}, 1719--1741 (2013)


\bibitem{SenTak2} Senti, S. and Takahasi, H.: Equilibrium measures for the H\'enon map
at the first bifurcation: uniqueness and geometric/statistical properties.
Ergodic Theory and Dynamical Systems {\bf 36}, 215--255 (2016)

\bibitem{Sin72} Sinai, 
Y.: Gibbs measures in ergodic theory. Uspekhi Mat. Nauk. {\bf 27},  21--64
(1972)

\bibitem{Tak11} Takahasi, H.: Abundance of non-uniform hyperbolicity 
in bifurcations of surface endomorphisms. Tokyo J. Math. {\bf 34}, 53--113 (2011)

\bibitem{Tak12} Takahasi, H.: Prevalent dynamics at the first bifurcation of 
H\'enon-like families. Commun. Math. Phys. {\bf 312}, 37--85 (2012)

\bibitem{Tak13} Takahasi, H.: Prevalence of non-uniform hyperbolicity at the first bifurcation of H\'enon-like families, submitted,
Available at {\sf http://arxiv.org/abs/1308.4199}




\bibitem{Tak15} Takahasi, H.: Equilibrium measures at temperature zero for H\'enon-like maps at the first bifurcation.
SIAM Journal on Applied Dynamical Systems. {\bf 15}, 106--124 (2016)







\bibitem{WanYou01} Wang, Q. D. and Young, L.-S.: Strange attractors with one
direction of instability. Commun. Math. Phys. {\bf 218}, 1--97 (2001)


\end{thebibliography}

\end{document}